\documentclass[twoside]{article}%

\usepackage{mathtools}
\usepackage{latexsym}
\usepackage{mathrsfs}
\usepackage{listings}
\usepackage{hyperref}
\usepackage{pgfplots}
\usepackage{appendix}
\hypersetup{
  unicode=true,
  pdfauthor={},
  pdftitle={Patterns of Non-Stationary Solutions},
  pdfsubject={},
  pdfkeywords={}; {}
}
\usepackage{pst-all}	%call the pstricks package
\usepackage{xcolor}
\usepackage{enumitem}
\usepackage{todonotes}
\usepackage{amssymb,bm}
\usepackage{amsmath}
\usepackage{amsfonts}
\usepackage{amsthm}
\usepackage{tikz}
\usepackage{epsfig}
\usepackage{verbatim}
\usepackage{float}
\usepackage{tabularx}
\usepackage{bm}
\usepackage{tikz-cd}

\providecommand{\U}[1]{\protect\rule{.1in}{.1in}}

\newcolumntype{Y}{>{\raggedleft\arraybackslash}X}
\def\bc{{\mathbb{C}}}

\def\bn{{\mathbb{N}}}

\def\br{{\mathbb{R}}}

\def\bz{{\mathbb{Z}}}

\def\br{\mathbb R}

\def\vs{\vskip.3cm}
\def\noi{\noindent}

\def\gdeg{G\text{\rm -deg}}

\def\Om{\Omega}

\def\sign{\text{\rm sign\,}}

\def\ker{\text{\rm Ker\,}}

\DeclareMathOperator{\id}{Id}

\newcommand\cU{\ensuremath{\mathcal U}}

\newrgbcolor{violet}{.6 .1 .8}
  \newrgbcolor{lightyellow}{1 1 .8}
  \newrgbcolor{lightblue}{.80 1 1}
  \newrgbcolor{mygreen}{0 .66 .05}
  \definecolor{mygreen}{rgb}{0,.66,.05}
  \definecolor{lightyellow}{rgb}{1,1,.80}
  \newrgbcolor{orange}{1 .6 0}
  \newrgbcolor{GreenYellow}{.85 1 .31}
  \newrgbcolor{Yellow}{1  1  0}
  \newrgbcolor{Goldenrod}{1  .90  .16}
  \newrgbcolor{Dandelion}{1  .71  .16}
  \newrgbcolor{Apricot}{1  .68  .48}
  \newrgbcolor{Peach}{1  .50  .30}
  \newrgbcolor{Melon}{1  .54  .50}
  \newrgbcolor{YellowOrange}{1  .58  0}
  \newrgbcolor{Orange}{1  .39  .13}
  \newrgbcolor{BurntOrange}{1  .49  0}
  \newrgbcolor{Bittersweet}{1.  .4300  .24}
  \newrgbcolor{RedOrange}{1  .23  .13}
  \newrgbcolor{Mahogany}{1.  .4475  .4345}
  \newrgbcolor{Maroon}{1.  .4084  .5376}
  \newrgbcolor{BrickRed}{1.  .3592  .3232}
  \newrgbcolor{Red}{1  0  0}
  \newrgbcolor{OrangeRed}{1  0  .50}
  \newrgbcolor{RubineRed}{1  0  .87}
  \newrgbcolor{WildStrawberry}{1  .04  .61}
  \newrgbcolor{CarnationPink}{1  .37  1}
  \newrgbcolor{Salmon}{1  .47  .62}
  \newrgbcolor{Magenta}{1  0  1}
  \newrgbcolor{VioletRed}{1  .19  1}
  \newrgbcolor{Rhodamine}{1  .18  1}
  \newrgbcolor{Mulberry}{.6668  .1180  1.}
  \newrgbcolor{RedViolet}{.9538  .4060  1.}
  \newrgbcolor{Fuchsia}{.5676  .1628  1.}
  \newrgbcolor{Lavender}{1  .52  1}
  \newrgbcolor{Thistle}{.88  .41  1}
  \newrgbcolor{Orchid}{.68  .36  1}
  \newrgbcolor{DarkOrchid}{.60  .20  .80}
  \newrgbcolor{Purple}{.55  .14  1}
  \newrgbcolor{Plum}{.50  0  1}
  \newrgbcolor{Violet}{.98 .15 .95}
  \newrgbcolor{RoyalPurple}{.25  .10  1}
  \newrgbcolor{BlueViolet}{.84  .38  .98}
  \newrgbcolor{Periwinkle}{.43  .45  1}
  \newrgbcolor{CadetBlue}{.38  .43  .77}
  \newrgbcolor{CornflowerBlue}{.35  .87  1}
  \newrgbcolor{MidnightBlue}{.4414  .9259  1.}
  \newrgbcolor{NavyBlue}{.06  .46  1}
  \newrgbcolor{RoyalBlue}{0  .50  1}
  \newrgbcolor{Blue}{0  0  1}
  \newrgbcolor{Cerulean}{.06  .89  1}
  \newrgbcolor{Cyan}{0  1  1}
  \newrgbcolor{ProcessBlue}{.04  1  1}
  \newrgbcolor{SkyBlue}{.38  1  .88}
  \newrgbcolor{Turquoise}{.15  1  .80}
  \newrgbcolor{TealBlue}{.1572  1.  .6668}
  \newrgbcolor{Aquamarine}{.18  1  .70}
  \newrgbcolor{BlueGreen}{.15  1  .67}
  \newrgbcolor{Emerald}{0  1  .50}
  \newrgbcolor{JungleGreen}{.01  1  .48}
  \newrgbcolor{SeaGreen}{.31  1  .50}
  \newrgbcolor{Green}{0  1  0}
  \newrgbcolor{ForestGreen}{.1992  1.  .2256}
  \newrgbcolor{PineGreen}{.3100  1.  .5575}
  \newrgbcolor{LimeGreen}{.50  1  0}
  \newrgbcolor{YellowGreen}{.56  1  .26}
  \newrgbcolor{SpringGreen}{.74  1  .24}
  \newrgbcolor{OliveGreen}{.6160  1.  .4300}
  \newrgbcolor{RawSienna}{.53  .28  .16}
  \newrgbcolor{Sepia}{1.  .7510  .70}
  \newrgbcolor{Brown}{.41  .25  .18}
  \newrgbcolor{TAN}{.86  .58  .44}
  \newrgbcolor{Gray}{1.  1.  1.}
  \newrgbcolor{Black}{1  1  1}
  \newrgbcolor{White}{1  1  1}
\definecolor{cadmiumgreen}{rgb}{0.0, 0.42, 0.24}
\definecolor{codegray}{rgb}{0.5,0.5,0.5}

\lstdefinestyle{mystyle}{
    commentstyle=\color{codegreen},
    keywordstyle=\color{magenta},
    numberstyle=\tiny\color{codegray},
    stringstyle=\color{codepurple},
    basicstyle=\ttfamily\footnotesize,
    breakatwhitespace=false,         
    breaklines=true,                 
    captionpos=b,                    
    keepspaces=true,                 
    numbers=left,                    
    numbersep=5pt,                  
    showspaces=false,                
    showstringspaces=false,
    showtabs=false,                  
    tabsize=2
}
  \lstdefinelanguage{GAP}{
    basicstyle=\ttfamily,
    keywords={true, false, function, return, fail, if, in, while, do, od, else, elif, fi, break, continue},
    keywordstyle=\color{NavyBlue}\bfseries,
    otherkeywords={% Operators
      >, <, ==
    },
    breaklines=true,      
    identifierstyle=\color{black},
    sensitive=True,
    comment=[l]{\#},
    commentstyle=\color{cadmiumgreen},
    stringstyle=\color{black},
    morestring=[b]',
    morestring=[b]"
  }
\newtheorem{theorem}{Theorem}[section]
\newtheorem{proposition}[theorem]{Proposition}
\newtheorem{lemma}[theorem]{Lemma}
\newtheorem{corollary}[theorem]{Corollary}

\newtheorem{remark}[theorem]{Remark}

\begin{document}

\title{Patterns of Non-Stationary Solutions to Symmetric Systems of Second-Order Differential Equations}

\author{
Ziad Ghanem\thanks{\small Department of Mathematical Sciences, University of Texas at Dallas, Richardson, TX 75080, USA
}}
%}
%}
\date{}

\maketitle

\begin{abstract}
We establish the existence of non-stationary solutions to a symmetric system of second-order autonomous differential equations. Our technique is based on the equivariant degree theory and involves a novel characterization of orbit types of maximal kind in the Burnside Ring product of a finite number of basic degrees for the group $O(2) \times \Gamma \times \mathbb Z_2$.
\end{abstract}

\noi \textbf{Mathematics Subject Classification:} Primary: 34C25, 37C81, 47H11, 55M25, 19A22

\medskip

\noi \textbf{Key Words and Phrases:}  symmetric equation, existence of solutions, equivariant
Leray-Schauder degree, nonlinear analysis, periodic solutions.

\section{Introduction} \label{sec:introduction}
The classical forced pendulum equation $\ddot x = \bm f(x)$ serves as a paradigmatic problem for validating novel techniques in nonlinear analysis (cf. \cite{Mahwin1}-\cite{Mahwin5}). More generally, dynamical systems described by second-order differential equations strike a delicate balance between simplicity and a capacity for capturing rich and intricate nonlinear behavior, making them ideal  candidates for the testing of analytical methods. On the other hand, autonomous differential equations with periodic boundary conditions always admit natural $SO(2)$ symmetries related to time shifting that make them particularly well-suited for a framework based on the equivariant degree theory. Second order autonomous differential equations not involving any first derivative terms, specifically, admit $O(2)$ symmetries, where the reflection generator acts by time reversal. When these equations are subjected to an odd forcing function and arranged in symmetric networks characterized by a finite group $\Gamma$, their full symmetry group becomes
\[
G:= O(2) \times \Gamma \times \bz_2.
\]
Our primary objective is to demonstrate how the $G$-equivariant Leray-Schauder degree can be used to establish the existence of non-stationary solutions to the following $\Gamma$-symmetric system of $p$-periodic, second order, autonomous equations:
\begin{align} \label{eq:system_periodic}
    \begin{cases}
        \ddot x(t) = f(x(t)) + Ax(t), \quad t \in \br, \;  x(t) \in V \\
        x(t) = x(t+ p), \; \dot x(t) = \dot x(t+p),
    \end{cases}
\end{align}
where $V := \br^N$ is an orthogonal $\Gamma$-representation, $A: V \rightarrow V$ is a symmetric $\Gamma$-equivariant matrix and $f: V \rightarrow V$ is a continuous function satisfying the conditions
\begin{enumerate}[label=($A_\arabic*$)]
    \item\label{A1} $f$ is $\Gamma$-equivariant, i.e. $f(\gamma x) = \gamma f(x)$ for all $\gamma \in \Gamma, \; t \in \br, \; x \in V$.
    \item\label{A2} $f$ is odd, i.e. $f(-x) = -f(x)$ for all $x \in V$.
    \item\label{A3} $f(x)$ is $o(|x|)$ as $x$ approaches $0$, i.e.
    \[
    \lim_{x \rightarrow 0} \frac{f(x)}{|x|} = 0.
    \]
    \item \label{A4} there exists a constant $M > 0$ such that for all $x \in V$ with $| x | \geq M$, one has
    \[
    f(x) \cdot x > 0.
    \]
\end{enumerate}
Conditions \ref{A1}--\ref{A2} ensure that the equations admit the symmetries of the product group $O(2) \times \Gamma \times \bz_2$, while conditions \ref{A3} and \ref{A4} imply the well-behavedness of our non-linearity at solutions to \eqref{eq:system_periodic} near zero and infinity, respectively. Specifically, \ref{A2} allows the equivariant degree to distinguish between stationary and non-stationary functions, \ref{A3} guarantees that the nonlinearity does not interact with the spectrum of $A$ and the Nagumo condition \ref{A4} is essential for establishing a-priori bounds on the magnitude of solutions to \eqref{eq:system_periodic}.
\vs
System \eqref{eq:system_periodic} is a natural model for the dynamics of a configuration of $N$ identical pendula, subjected to nonlinear forcing $f$ and with coupling relations specified by the matrix $A$ and the group $\Gamma$. In this context, the vector $x(t) := (x_1(t), x_2(t), \ldots, x_N(t))$ represents the angles
of the pendula from the vertical position, $\dot x(t) := (\dot x_1(t), \dot x_2(t), \ldots, \dot x_N(t))$ the angular velocities and $\ddot x(t) := (\ddot x_1(t), \ddot x_2(t), \ldots, \ddot x_N(t))$ the angular accelerations at any time $t > 0$. 
\vs
Methods based on the Leray-Schauder degree have been used to solve a wide variety of second order differential equations \cite{Amster, Bebernes, Gaines, Hartman, Knobloch, Mahwin3, Mahwin5}. However, the inability of the Leray-Schauder degree to distinguish between periodic solutions that differ by a fixed time shift and differentiate constant solutions from non-constant solutions have hindered its suitability for autonomous differential equations.
The effectiveness of the $G$-equivariant Leray-Schauder degree in detecting non-constant solutions to the system \eqref{eq:system_periodic} for several explicit choices of the symmetry group $\Gamma$ has already been demonstrated by Garcia-Azpetia et al. in the article \cite{Carlos}. More generally, the $G$-equivariant Leray-Schauder degree has been applied to problems admitting appropriate spatial-temporal symmetries \cite{BalBurnett, BalChen, Duan, Eze, Ghanem, Krawcewicz}.
Compared with other techniques to solve systems similar to \eqref{eq:system_periodic}, the equivariant degree approach requires relatively minimal regularity assumptions on the nonlinearity. For example, existence results were obtained in \cite{Amann} by topological index argument and in \cite{Chang} using Morse theory framework, under the assumption that the right-hand-side interacts with  only finitely many eigenvalues of the differential operator.
\vs
A prerequisite for the application of any $G$-equivariant degree-theory based argument to \eqref{eq:system_periodic} is the reformulation of our system as a fixed point equation in an appropriate functional $G$-space with a nonlinear operator in the form of a $G$-equivariant compact perturbation of identity. In \cite{Carlos}, Garcia-Azpetia et al. prove that the problem of finding $p$-periodic solutions to \eqref{eq:system_periodic} is equivalent to the problem of finding zeros of the operator
\[
\mathscr F: \mathscr H \rightarrow \mathscr H, \quad \mathscr F(u) := u - \mathscr L^{-1}(\beta^2 f(u) + \beta^2Au - u), 
\]
where $\beta := \frac{p}{2 \pi}$, $\mathscr H:= H^2_{2\pi}([0,2\pi];V)$ and $\mathscr L: \mathscr H \rightarrow \mathscr H$ is the differential operator $\mathscr L u := \ddot u - u$. 
Under the conditions \ref{A1}---\ref{A4} and provided that the linearization $ D\mathscr F(0): \mathscr H \rightarrow \mathscr H$ is an isomorphism, Garcia-Azpetia et al. additionally demonstrate 
that: 
\begin{enumerate}
    \item[$\rm (i)$] $\mathscr F$ is a $G$-equivariant completely continuous field with respect to the natural isometric action of $G$ on $\mathscr H$;
    \item[$\rm (ii)$] all of the non-trivial solutions to \eqref{eq:system_periodic} can be confined to an annulus of the form $\Om := B_R(\mathscr H) \setminus B_{\varepsilon}(\mathscr H) \subset \mathscr H$ for some sufficiently small choice of $\varepsilon > 0$ and sufficiently large choice of $R > 0$;
    \item[$\rm (iii)$] $(\mathscr F,\Om)$ constitutes an admissible $G$-pair, such that the existence of non-trivial solutions to \eqref{eq:system_periodic} is equivalent to non-triviality of the degree invariant 
    \[
\gdeg(\mathscr F,\Om) = (G) - \gdeg(D \mathscr F(0), B_1(\mathscr H)),
\]
where $\gdeg$ is the $G$-equivariant Leray-Schauder degree and $(G) \in A(G)$ is the unit element of the Burnside ring (cf. Appendix \ref{sec:appendix} for the definition of the Burnside ring $A(G)$, the definition of $G$-admissibility/the set of all admissible $G$-pairs $\mathscr M^G$ and for the construction of the degree $\gdeg:\mathscr M^G \rightarrow A(G)$).
\end{enumerate}
Together with a direct application of the existence property of the $G$-equivariant Leray-Schauder degree (cf. Appendix \ref{sec:appendix}), the above considerations suggest the following sufficient condition for the existence of a non-trivial solution to \eqref{eq:system_periodic}:
\begin{lemma}\label{lemm:sufficient_condition}
If for some orbit type $(H) \in \Phi_0(G) \setminus \{(G) \}$ one has
\[
\operatorname{coeff}^{H}\left( \gdeg(D\mathscr F(0), B(\mathscr H)) \right) \neq 0,
\]
(cf. Appendix \ref{sec:appendix} for definition of the coefficient operator $\operatorname{coeff}^{H}: A(G) \rightarrow \bz$) then there exists a function $u \in \mathscr H \setminus \{0\}$ satisfying \eqref{eq:system} with an isotropy subgroup $G_u \leq G$ satisfying $(G_u) \geq (H)$.   
\end{lemma}
In this way, the problem of finding $p$-periodic solutions to \eqref{eq:system_periodic} has been reformulated as the problem of computing the coefficients of non-unit orbit types (i.e. orbit types $(H) \in \Phi_0(G) \setminus \{(G)\}$) in the Burnside ring element $\gdeg(D\mathscr F(0), B(\mathscr H))$. Moreover, since the $G$-equivariant Leray-Schauder degree provides a full equivariant classification to the solution set of $\mathscr F(u) = 0, \; u \in \mathscr H$, Lemma \ref{lemm:sufficient_condition} represents a framework for identifying the exact spatio-temporal symmetries of any predicted solution. In order to take advantage of the full power of the $G$-equivariant Leray-Schauder degree for this purpose, we must first identify the irreducible representations of $G$ and describe the $G$-isotypic decomposition of $\mathscr H$.  
\vs
Without specifying the group $\Gamma$, we assume that a complete list of the irreducible $\Gamma$-representations $\{ \mathcal V_j \}_{j=0}^r$ is made available. As a $\Gamma$-representation, the space $V$ is also a natural $\Gamma \times \mathbb{Z}_2$-representation with the $\Gamma \times \mathbb{Z}_2$-isotypic decomposition
\begin{align*}
  V = V_0 \oplus V_1 \oplus \cdots \oplus V_r,      
\end{align*}
%\label{Gamma_Z2_isotypic_decomposition}
where each $\Gamma \times \mathbb{Z}_2$-isotypic component $V_j$ is modeled on the irreducible $\Gamma \times \mathbb{Z}_2$-representation $\mathcal V_j^-$ (here the superscript is meant to indicate that $\mathcal V_j$ has been equipped with antipodal $\mathbb{Z}_2$-action) in the sense that $V_j$ is equivalent to the direct sum of some finite number of copies of $\mathcal V_j^-$, i.e.
\[
V_j \simeq \mathcal V_j^- \oplus \cdots \oplus \mathcal V_j^-.
\]
The exact number of irreducible $\Gamma$-representations $\mathcal V_j^-$ `contained' in the $\Gamma \times \mathbb{Z}_2$-isotypic component $V_j$ is called the \textit{$\mathcal V_j$-isotypic multiplicity} of $V$ and is calculated according to the ratio:
\begin{align}\label{def:isotypic_multiplicity}
  m_j:= \dim V_j / \dim \mathcal{V}_j^-, \quad j \in \{0,1,\ldots,r\}.    
\end{align}
Since $A:V \rightarrow V$ is $\Gamma$-equivariant, one has 
\[
\sigma(A) = \bigcup_{j=0}^r \sigma(A_j), \quad A_j := A|_{V_j}:V_j \rightarrow V_j.
\]
For the sake of simplicity, we choose to sidestep a fixed point argument made by
Garcia-Azpetia et al. in \cite{Carlos} by enforcing the following non-degeneracy assumption:
\begin{enumerate}[label=($A_0$)]
    \item\label{A0} for each $\Gamma$-isotypic component $V_j$, there exists a number $\mu_j \in \br$ with $\mu_j \neq - m^2/\beta^2$ for all $m \in \bn$ such that
    \[
    A_{j} = \mu_j \id|_{V_j}: V_j \rightarrow V_j.
    \]
\end{enumerate}
On the other hand, for each $m \in \bn$, we denote by $\mathcal W_m \simeq \mathbb{C}$ the irreducible $O(2)$-representation equipped with the $m$-folding $O(2)$-action
\begin{align*}
e^{i \theta}w:= e^{i m \theta}w, \quad \kappa w:= \bar{w},  \quad e^{i\theta},\kappa \in O(2),  \; w \in \mathcal W_m,
\end{align*}
(here, $\kappa$ indicates the reflection generator in $O(2)$) and by $\mathcal W_0 \simeq \mathbb{R}$ the irreducible $O(2)$-representation on which $O(2)$ acts trivially such that the space $\mathscr H$ admits the $O(2)$-isotypic decomposition 
\begin{align*}
 \mathscr H := \overline{\bigoplus\limits_{m=0}^\infty \mathscr H_m}, \quad \mathscr H_m  := \{ u \in \mathscr H : u(t) = \cos(mt)a + \sin(mt)b, \; a,b \in V\},
\end{align*} 
where
\[
\mathscr H_m \simeq \mathcal W_m \otimes V, \quad m \in \bn \cup \{0\}.
\]
The $G$-isotypic decomposition of $\mathscr H$ can now be expressed in terms of $G$-isotypic components $\mathscr H_{m,j} \subset \mathscr H$ modeled on the irreducible $G$-representations
\begin{align*}
    \mathcal V_{m,j} := \mathcal W_m \otimes \mathcal{V}_j^-, \quad m \in \bn \cup \{0\}, \; j \in \{0,1,\ldots,r\},
\end{align*}
as follows
\begin{align} 
\label{G-isotypic_decomposition}
   \mathscr H = \bigoplus_{j=0}^r \overline{\bigoplus\limits_{m=0}^\infty \mathscr H_{m,j}}, 
\end{align}
where
\[
\mathscr H_{m,j} := \{ u \in \mathscr H : u(t) = \cos(mt)a + \sin(mt)b, \; a,b \in V_j\}.
\]
We are now in a position to present our main existence results:
\begin{theorem}\label{thm:main_theorem}
Take $m > 0$ and let $(H)$ be a maximal element in the isotropy lattice $\Phi_0(G; \mathscr H_{m} \setminus \{0\})$. Under the assumptions \ref{A0}--\ref{A4}, if the matrix $A:V \rightarrow V$ admits an {\bf odd number} of eigenvalues $\mu_j \in \sigma(A)$ with odd isotypic multiplicity $m_j$ (cf. \eqref{def:isotypic_multiplicity}) satisfying
\[
\mu_j < -\left(\frac{m}{\beta}\right)^2 \text{  and  } 2 \nmid \dim \mathcal V_{m,j}^H,
\]
then there exists a non-stationary solution $u \in \mathscr H \setminus \{0\}$ to the system \eqref{eq:system_periodic} with an isotropy subgroup $G_u \leq G$ satisfying $(G_u) \geq (H)$. 
\end{theorem}
\begin{remark} \rm
In the field of equivariant bifurcation theory, the Ize Conjecture (IC) refers to the proposition that every irreducible representation of a compact Lie group admits at least one subgroup with an odd dimensional fixed point space. Although the IC was shown to be false by Lauterbach et al. in \cite{Lauter} via counterexample, the conjecture holds in many cases. For example, in Section \ref{sec:special_case}, we use G.A.P. to numerically verify that the irreducible representations of the group $G = O(2) \times \Gamma \times \bz_2$, with the assignment $\Gamma = D_8$, satisfy the IC.
\end{remark}
The remainder of this paper is organized as follows: 
\vs
In Section \ref{sec:degree_setting}, we prepare the problem \eqref{eq:system_periodic} for application of the $G$-equivariant Leray-Schauder degree. This involves 
a functional reformulation in an appropriate Sobolev space, allowing us to derive a practical formula for a degree invariant, in terms of the Burnside ring product of a finite number of basic degrees defined on the irreducible $G$-representations, whose non-triviality is equivalent to the existence of non-trivial solutions to \eqref{eq:system_periodic}. Using this framework, we analyze the maximal elements in the isotropy lattices $\Phi_0(G; \mathscr H_{m,j}\setminus \{0\})$, called {\it orbit types of maximal kind} in $\Phi_0(G)$. The significance of these orbit types is two-fold: first, the coefficients corresponding to orbit types of maximal kind in the degree invariant are computationally accessible and second, the non-triviality of these coefficients guarantees the existence of non-stationary solutions to \eqref{eq:system_periodic}.
In Section \ref{sec:generator_classification}, we present a classification for the generators of the Burnside ring $A(G)$ according to their lattice relations with respect to the $s$-folding homomorphism \eqref{def:sfolding_burnside} and employ this classification, in Section \ref{sec:char_maxorbtyps}, to derive a practical formula for characterizing the non-triviality of an orbit type of maximal kind in the product of a finite number of basic degrees. Specifically, we find that the contribution of each eigenvalue $\mu_{m,j} \in \sigma(D \mathscr F(0))$ to the appearance of an orbit type of maximal kind $(H)$ in the degree invariant $\gdeg(\mathscr F,\Om)$ depends on $(\rm i)$ the sign of $\mu_{m,j}$ and $(\rm ii)$ the parity of the dimension of the corresponding $H$-fixed-point space $\mathcal V_{m,j}^{H}$. Finally, in Section \ref{sec:n_pendula}, we consider the implications of our main result for dihedral symmetry group $\Gamma = D_N$. For the special case of $D_8$ coupling symmetries, we use GAP to numerically verify $(\rm i)$ the pairwise disjointness of the maximal orbit type sets in each isotropy lattice $\Phi_0(G; \mathscr H_{m,j} \setminus \{0\})$ and $(\rm ii)$ the odd-dimensionality of the fixed point sets $\mathcal V_{m,j}^H$ for every maximal element in $\Phi_0(G; \mathscr H_{m,j} \setminus \{0\})$, allowing us to reformulate Theorem \ref{thm:main_theorem} in Section \ref{sec:special_case} as follows:
\begin{proposition}\label{prop:special_case}
Let $\Gamma = D_8$ and take any fixed $m > 0$.
Under the assumptions \ref{A0}--\ref{A4}, if the $j$-th eigenvalue of the matrix $A$ satisfies
\[
\mu_j < -\left(\frac{m}{\beta}\right)^2,
\]
then, corresponding to every maximal element $(H)$ in the isotropy lattice $\Phi_0(G; \mathscr H_{m,j} \setminus \{0\})$, there exists a non-stationary solution $u \in \mathscr H \setminus \{0\}$ to the system \eqref{eq:system_periodic} with an isotropy subgroup $G_u \leq G$ satisfying $(G_u) \geq (H)$. 
\end{proposition}
For convenience, the Appendices include an explanation of notations used, and a brief introduction to the $G$-equivariant Leray-Schauder degree. Readers who are interested in a deeper exposition of these topics are referred to \cite{book-new,AED}.
\section{Setting for the Leray-Schauder $G$-Equivariant Degree} \label{sec:degree_setting}
Notice that, with the substitutions $\beta := p/2\pi$ and $u(t):=x(\beta t)$, the system \eqref{eq:system_periodic} becomes
\begin{align} \label{eq:system}
    \begin{cases}
        \ddot u(t) = \beta^2 f(u(t)) + \beta^2 Au(t) , \quad t \in \br, \; u(t) \in V \\
        u(t) = u(t+ 2\pi), \; \dot u(t) = \dot u(t+2 \pi).
    \end{cases}
\end{align}
Clearly, any $2 \pi$-periodic solution $u(t) : \br \rightarrow V$ for the system \eqref{eq:system} corresponds to the $p$-periodic solution $x(t) = u(t/\beta)$ for our original equation.
\vs
Following \cite{Carlos}, we denote by $\mathscr H$ the Sobolev space of $2 \pi$-periodic, $V$-valued functions
\[
\mathscr H := \{ u : \br \rightarrow V : u(0) = u(2 \pi), \; u|_{[0,2\pi]} \in H^2(S^1;V) \},
\]
equipped with the standard inner product
\[
\langle u,v \rangle := \int_{0}^{2 \pi} (\dot u(t) \bullet \dot v(t) + u(t) \bullet v(t)) dt, \quad u,v \in \mathscr H,
\]
and the isometric $G$-action
\begin{align*} %\label{def:G_action}
(e^{i\theta}, \gamma, \pm 1)u(t) := \pm \gamma u(t + \theta), \quad 
(\kappa, \gamma, \pm 1)u(t) := \pm \gamma u(-t), \quad u \in \mathscr H, \; t \geq 0, 
\end{align*}
where $(e^{i\theta},\gamma,\pm 1), (\kappa,\gamma,\pm 1) \in O(2) \times \Gamma \times \bz_2$. We also consider the differential operator
\begin{align*}
%\label{def:L_operator}
    \mathscr L: \mathscr H \rightarrow \mathscr E, \quad \mathscr Lu := \ddot u - u,
\end{align*}
the Nemytskii operator
\begin{align*}
%\label{def:N_operator}
    N: L^2(S^1;V) \rightarrow L^2(S^1;V), \quad (Nu)(t) :=  f(u(t)),
\end{align*}
and the Banach embedding 
\begin{align*}
%\label{def:j_operator}
    j: \mathscr H \rightarrow L^2(S^1;V), \quad j(u)(t):= u(t).
\end{align*}
Since $\mathscr L:\mathscr H \rightarrow L^2(S^1;V)$ is a linear isomorphism, $N:L^2(S^1;V) \rightarrow L^2(S^1;V)$ is continuous and $j: \mathscr H \rightarrow L^2(S^1;V)$ is compact, the operator $\mathscr F: \mathscr H \rightarrow \mathscr H$ given by
\begin{align} \label{def:F_operator}
    \mathscr F(u) := u - \mathscr L^{-1}(\beta^2 N(j(u)) + \beta^2 A j(u) - j(u)),
\end{align}
is a compact perturbation of the identity on $\mathscr H$ and the system \eqref{eq:system} is equivalent to the operator equation
\begin{align} \label{eq:operator_equation}
    \mathscr F(u) = 0,
\end{align}
in the sense that a function $u \in \mathscr H$ is a solution to \eqref{eq:system} if and only if it satisfies \eqref{eq:operator_equation}. In what follows, \eqref{eq:system} will be called the {\it operator equation} associated with the system \eqref{eq:system}.
\vs
Gracia-Azpetia et al. demonstrate the applicability of the $G$-equivariant Leray-Schauder degree to \eqref{eq:operator_equation} by proving that the assumptions \ref{A1}--\ref{A4} imply that the nonlinear operator \eqref{def:F_operator} is a completely continuous $G$-equivariant field, differentiable at the origin $0 \in \mathscr H$ with
\[
\mathscr Au := \id - \mathscr L^{-1}(\beta^2 Au - u), \quad \mathscr A := D \mathscr F(0): \mathscr H \rightarrow \mathscr H.
\]
\subsection{Towards a Computational Formula for $\gdeg(\mathscr A, B(\mathscr H))$} 
\label{sec:computational_formula}
In order to effectively make use Lemma \ref{lemm:sufficient_condition} to determine the existence of non-trivial solutions to \eqref{eq:system_periodic}, we must derive a practical formula for the computation of the local bifurcation invariant $\gdeg(\mathscr A, B(\mathscr H)) \in A(G)$. Our first step in this direction will be to take advantage of the $G$-isotypic decomposition \eqref{G-isotypic_decomposition} to collect spectral data related to the $G$-equivariant linear operator $\mathscr A: \mathscr H \rightarrow \mathscr H$. 
\vs
By Schur's Lemma, one has
\[
\mathscr A(\mathscr H_{m,j}) \subset \mathscr H_{m,j}, \quad m \in \bn \cup \{0\}, \; j \in \{0,1,\ldots,r\},
\]
such that, adopting the notation
\begin{align*}
    \mathscr A_{m,j} := \mathscr A|_{\mathscr H_{m,j}}: \mathscr H_{m,j} \rightarrow \mathscr H_{m,j},
\end{align*}
the spectrum of $\mathscr A$ admits the following decomposition
\begin{align*} %\label{mathscrA_spectral_decomposition1}
 \sigma (\mathscr A )=  \bigcup\limits_{m=0}^{\infty} \bigcup\limits_{j=0}^{r} \sigma(\mathscr A_{m,j}).
\end{align*}
Under assumption \ref{A0}, the spectrum of each matrix $\mathscr A_{m,j}$ consists of the nonzero eigenvalue
\begin{align}\label{def:operator_eigenvalues}
\mu_{m,j} := 1 + \frac{\beta^2 \mu_j - 1}{1 + m^2} = \frac{m^2 + \beta^2 \mu_j}{1+m^2},    
\end{align}
with multiplicity $m_j \in \bn$ (cf. discussion of isotypic multiplicity \eqref{def:isotypic_multiplicity}).
\vs
The product property of the $G$-equivariant Leray-Schauder degree (cf. Appendix \ref{sec:appendix}) permits us to express the degree $\gdeg(\mathscr A, B(\mathscr H))$ in terms of a Burnside ring product of the $G$-equivariant Leray-Schauder degrees of the various restrictions $\mathscr A_{m,j}: \mathscr H_{m,j} \rightarrow \mathscr H_{m,j}$ of the $G$-equivariant linear isomorphism $\mathscr A: \mathscr H \rightarrow \mathscr H$ to the $G$-isotypic components $\mathscr H_{m,j}$ on their respective open unit balls $B(\mathscr  H_{m,j}):= \{ u \in \mathscr H_{m,j} \; : \; \| u \|_{\mathscr H} < 1 \}$ as follows
\begin{equation}\label{eq:gdegA_product_property_decomp}
  \gdeg(\mathscr A, B(\mathscr H)) = \prod\limits_{j=0}^r \prod\limits_{m=0}^\infty \gdeg( \mathscr A_{m,j}, B(\mathscr H_{m,j})).  
\end{equation}
On the other hand, each degree $\gdeg( \mathscr A_{m,j}, B(\mathscr H_{m,j}))$ is fully specified by the $\mathcal V_j$-isotypic multiplicities $\{ m_j \}_{j=0}^r$ together with the real spectra of $\mathscr A_{m,j}$ according to formula:
\begin{equation}\label{eq:negative_spectrum_gdegA}
\gdeg(\mathscr A_{m,j}, B( \mathscr H_{m,j})) =
\begin{cases}
    (\deg_{\mathcal V_{m,j}})^{m_j} \quad & \text{ if } \mu_{m,j} < 0; \\
    (G) \quad & \text{ otherwise,}
\end{cases}    
\end{equation}
where $\deg_{\mathcal V_{m,j}} \in A(G)$ is the basic degree associated with the irreducible $G$-representation $\mathcal V_{m,j}$ (cf. Appendix \ref{sec:appendix}) and $(G) \in A(G)$ is the unit element of the Burnside 
Ring. In addition, since each basic degree $\deg_{\mathcal V_{m,j}}$ is involutive in the Burnside ring (cf. Appendix \ref{sec:appendix}), one has
\begin{equation}\label{eq:involutive_bdegs_gdegA}
  (\deg_{\mathcal V_{m,j}})^{m_j} =
\begin{cases}
    \deg_{\mathcal V_{m,j}} \quad & \text{ if } 2 \nmid m_j; \\
    (G) \quad & \text{ otherwise.}
\end{cases}   
\end{equation}
Putting together \eqref{eq:gdegA_product_property_decomp}, \eqref{eq:negative_spectrum_gdegA} and \eqref{eq:involutive_bdegs_gdegA}, we introduce some notations to keep track of the indices 
\begin{align*}
%\label{index_1}
   \Sigma := \left\{ (m,j) : m \in \mathbb{N} \cup \{0\}, \; j \in \{0,1,\ldots,r \} \right\},
\end{align*}
which contribute non-trivially to $\gdeg(\mathscr A, B(\mathscr H))$. To begin, the \textit{negative} spectrum of $\mathscr A: \mathscr H \rightarrow \mathscr H$ is accounted for with the index set
\begin{align*}
%\label{index_1}
   \Sigma_{-} := \left\{ (m,j) \in \Sigma :  \mu_{m,j} < 0 \right\}.
\end{align*}
Combining this with formulas \eqref{eq:gdegA_product_property_decomp} and \eqref{eq:negative_spectrum_gdegA} yields
\begin{align*}
    \gdeg(\mathscr A, B(\mathscr H)) \; = \prod\limits_{(m,j) \in \Sigma_{-}} (\deg_{\mathcal V_{m,j}})^{m_j}.
\end{align*}
This computation can be further reduced by discarding the indices $(m,j) \in \Sigma$ associated with even $\mathcal V_j$-isotypic multiplicities $m_j$, whose corresponding basic degrees contribute trivially to the Burnside product \eqref{eq:gdegA_product_property_decomp}. Specifically, we put
\begin{align*}
%\label{def:index_set}
\Sigma_0 := \left\{ (m,j) \in \Sigma_{-} \; : \; 2 \nmid m_j \right\},
\end{align*}
which, together with \eqref{eq:involutive_bdegs_gdegA}, permits us to express $\gdeg(\mathscr A, B(\mathscr H))$ as the Burnside ring product of a finite number of unique basic degrees, as follows
\begin{align} \label{eq:final_computation_gdegA}
\gdeg(\mathscr A, B(\mathscr H)) \; = \prod\limits_{(m,j) \in \Sigma_0} \deg_{\mathcal V_{m,j}}.
\end{align}
\section{Orbit Types of Maximal Kind} \label{sec:orbtyps_maximal}
It is always possible to study the non-triviality of the product of a finite number of Burnside ring elements (such as the product of basic degrees in the expression \eqref{eq:final_computation_gdegA}) as a problem concerning the non-triviality of the product of a finite number of generator pairs $(K),(L) \in \Phi_0(G)$. Keeping in mind that the non-triviality of 
\[
(K) \cdot (L) \in A(G),
\]
is equivalent to the existence of an orbit type $(H) \in \Phi_0(G) \setminus  \{(G) \}$ with $\operatorname{coeff}^{H}((K) \cdot (L)) \neq 0$, notice that the recurrence formula for the Burnside ring product (cf. Appendix \ref{sec:appendix})
\begin{align*} 
 &C^H(K,L) = \\ &\quad\frac{n(H,L) \vert W(L) \vert n(H,K) \vert W(K) \vert - \sum_{(\tilde H) > (H)} C^{\tilde H}(K,L) n(H,\tilde H) \vert W(\tilde H) \vert}{\vert W(H) \vert},   
\end{align*}
(here, we have temporarily adopted the notation $C^H(K,L):= \operatorname{coeff}^{H} ((K) \cdot (L))$ due to constraints on space) is only useful if one first identifies the set of all orbit types $(\tilde H) \in \Phi_0(G) \setminus \{(G)\}$ with $(\tilde H) > (H)$. Therefore, computation of the coefficients $\operatorname{coeff}^{H}( (K) \cdot (L))$ is generally impractical except for those orbit types which are maximal in $\Phi_0(G)$, i.e. for those orbit types $(H) \in \Phi_0(G)$ where $(\tilde H) > (H)$ implies $(\tilde H) = (G)$, in which case the recurrence formula simplifies to 
\[
\operatorname{coeff}^{H}((K) \cdot (L)) = \frac{n(H,L) \vert W(L) \vert n(H,K) \vert W(K) \vert}{\vert W(H) \vert}.
\]
\begin{remark} \rm \label{rm:nonconstant_solutions}
One of the principal difficulties when working with a product group of the form $G = O(2) \times \Gamma \times \bz_2$ is that the conjugacy class of the subgroup $O(2)\times \{e_\Gamma\} \times \{1\} \leq G$ is always a generator of the Burnside ring $A(G)$. Since any orbit type $(H)$ which is maximal in the isotropy lattice $\Phi_0(G;\mathscr H \setminus \{0\})$ must also be maximal in the sublattice $\Phi_0(G;\mathscr H_{0} \setminus \{0\})$, this implies that any non-zero function $u \in \mathscr H$ with an isotropy group $G_u \leq G$ such that $(G_u)$ is a maximal element in $\Phi_0(G;\mathscr H \setminus \{0\})$ must be $O(2)$-invariant. Therefore, any map $u \in \mathscr H$ which is both non-trivial and non-constant has an isotropy group $G_u \leq G$ with a conjugacy class that is maximal in $\Phi_0(G;\mathscr H_{m} \setminus \{0\})$ for some positive $m$. 
\end{remark}
Motivated by Remark \ref{rm:nonconstant_solutions}, the conjugacy class of a subgroup $H \leq G$ is called an {\it orbit type of maximal kind} if it is a maximal element in $\Phi_0(G; \mathscr H_m \setminus \{0\})$ for some $m > 0$. Moreover, we denote by $\mathfrak M$ the set of all orbit types of maximal kind, by $\mathfrak M_m$ the set of maximal elements in $\Phi_0(G;\mathscr H_{m} \setminus \{0\})$ and by $\mathfrak M_{m,j}$ the set of orbit types $\mathfrak M_m  \cap \Phi_0(G; \mathscr H_{m,j} \setminus \{0\})$. Since each orbit type of maximal kind $(H) \in \mathfrak M_m$ is contained in the sublattice of orbit types $\Phi_0(G; \mathscr H_{m,j} \setminus \{0\})$ for at least one $j \in \{0,1,\ldots,r\}$, we have
\begin{align*}
 \mathfrak M = \bigcup_{m > 0} \mathfrak M_m, \quad   \mathfrak M_m = \bigcup_{j=0}^r \mathfrak M_{m,j},
\end{align*}
corresponding to the $\Gamma \times \mathbb{Z}_2$-isotypic decomposition of the $m$-th $O(2)$-isotypic component of $\mathscr H$:
\begin{align*}
    \mathscr H_m = \mathscr H_{m,1} \oplus \mathscr H_{m,2} \oplus \cdots \oplus \mathscr H_{m,r}.
\end{align*}
At this point, we can strengthen the  sufficient condition for the existence of a non-trivial solution to \eqref{eq:system} stated in Lemma \ref{lemm:sufficient_condition} to a sufficient condition for the existence of a non-stationary solution:
\begin{lemma}\label{lem:sufficient_condition_stronger}
If for some orbit type of maximal kind $(H) \in \mathfrak M$ one has
\[
\operatorname{coeff}^{H}\left( \gdeg(\mathscr A, B(\mathscr H)) \right) \neq 0,
\]
then there exists a non-stationary function $u \in \mathscr H$ satisfying \eqref{eq:system}.   
\end{lemma}

\subsection{The $s$-Folding Homomorphism}\label{sec:sfolding}
In this section, we introduce the concept of $s$-folding, which plays a key role in our study of orbit types of maximal kind. For each $s \in \mathbb{N}$, we define the \textit{$s$-folding homomorphism} by
\begin{align*} 
%\label{def:sfolding_0}
\phi_s: O(2) \rightarrow O(2)/\mathbb{Z}_s \simeq O(2), \quad    \phi_s(e^{i \theta}) := e^{is \theta}, \quad \phi_s(\kappa e^{i \theta}) := \kappa e^{is \theta}.
\end{align*}
There is a natural extension of $\phi_s$ to the Lie group homomorphism $\psi_s: O(2) \times \Gamma \times \mathbb{Z}_2 \rightarrow  O(2) \times \Gamma \times \mathbb{Z}_2$ given by
\begin{align} \label{def:sfolding}
    \psi_s(e^{i \theta},\gamma, \pm 1) := (\phi_s(e^{i \theta}), \gamma, \pm 1), \quad \psi_s(\kappa e^{i \theta},\gamma, \pm 1) := (\phi_s(\kappa e^{i \theta}), \gamma, \pm 1).
\end{align}
In turn, \eqref{def:sfolding} induces the Burnside ring homomorphism $\Psi_s: A(G) \rightarrow A(G)$ defined as follows
\begin{align} 
\label{def:sfolding_burnside}
\Psi_s(H) := ({}^{s}H), \quad {}^{s}H := \psi^{-1}_s(H).
\end{align}
Clearly, one has the following relation between the sets $\mathfrak M_m$ and $\mathfrak M_{sm}$ for all $s > 0$ and $m \geq 0$
\begin{align*} 
%\label{def:sfolding_2}
\Psi_s(\mathfrak M_m) = \mathfrak M_{sm}.
\end{align*}
Zalman et al. hypothesize in \cite{BalChen}---where it serves as an essential component in the proof of their main result---that every orbit type $(H) \in \Phi_0(G; \mathscr H_m) \setminus \{(G)\}$ satisfies the relation 
\begin{align} \label{rel:conjecture}
 ({}^{s} H) > (H) \text{ for all } s \in \bn.   
\end{align}
In this paper, we demonstrate that the relation \eqref{rel:conjecture} does not hold universally. Instead, we prove that the relation between $(H)$ and $({}^{s} H)$ depends both on $(H) \in \Phi_0(G; \mathscr H_m)$ and $s \in \bn$. 
\section{A Classification of Generators for the Burnside ring Associated with the Group $G = O(2) \times \Gamma \times \mathbb{Z}_2$.}\label{sec:generator_classification}
In order to proceed with our computations in the Burnside ring, we must employ the convention of \textit{amalgamated notation}: a shorthand for the specification of subgroups in a product group, first considered by Balanov et al. in \cite{AED}.
It is a well known consequence of Goursat's Lemma (cf. \cite{Goursat})
that any closed subgroup in the product group $G = O(2) \times \Gamma \times \bz_2$ can be identified, up to its conjugacy class in $\Phi_0(G)$, with a quintuple $(K_O, K_\Gamma, L, \varphi_O, \varphi_\Gamma)$ consisting of two subgroups $K_O \leq O(2)$, $K_\Gamma \leq \Gamma \times \mathbb{Z}_2$, a group $L$ and a pair of epimorphisms $\varphi_O: K_O \rightarrow L$, $\varphi_\Gamma: K_\Gamma \rightarrow L$, as follows
\begin{align} \label{def:amalgamated_notation}
    K_O {}^{\varphi_O}\times^{\varphi_\Gamma}_{L} K_\Gamma := \{ (x,y) \in K_O \times K_\Gamma \; : \; \varphi_O(x) = \varphi_\Gamma(y) \}.
\end{align}
Let $H$ be a closed subgroup in $G$ with the {\it amalgamated decomposition} \eqref{def:amalgamated_notation} and denote by $r(H) \in SO(2)$ the rotation generator in $K_O$. Since the order of $\varphi_O(r(H)) \in L$ is invariant under conjugation of $r(H)$ in $K_O$, the following map is well-defined:
\begin{align} \label{def:amalgamated_order_map}
  \mathfrak m : \Phi_0(G) \rightarrow \bn, \quad  \mathfrak m(H):= |\varphi_O(r(H))|.
\end{align}
Given any orbit type $(H) \in \Phi_0(G)$ and a number $s_0 \in \bn$, we will demonstrate how the map \eqref{def:amalgamated_order_map} can be used to describe the set
\begin{align}\label{def:amalgamated_order_set}
\left\{ s \in \bn : ({}^{s}H) \leq ({}^{s_0}H) \right\}.
\end{align}
In turn, for any orbit type of maximal kind $(H) \in \mathfrak M$ and for any pair of generators $(K),(L) \in \Phi_0(G)$ we can use the set \eqref{def:amalgamated_order_set} to characterize the coefficients of $({}^{s}H)$ for all $s \in \bn$ in the Burnside ring product  
\[
(K) \cdot (L) \in A(G).
\]
\begin{remark} \rm
For any orbit type of maximal kind $(H) \in \mathfrak M$, the relation $(H) < (\tilde H)$ implies $(\tilde H) = (G)$ or $(\tilde H) = ({}^{s}H)$ for some $s \in \bn$. Therefore, since satisfaction of the relations $(H) < (K)$ and $(H) < (L)$ is a prerequisite for the non-triviality $\operatorname{coeff}^{H}((K) \cdot (L)) \neq 0$, we can restrict our focus to the coefficients of $({}^{s}H)$ for all $s \in \bn$ in Burnside ring products of the form 
\begin{align*}
%\label{def:amalgamated_order_coefficients}
({}^{s_0}H) \cdot ({}^{s_1}H) \in A(G), \quad s_0, s_1 \in \bn.
\end{align*}   
\end{remark}
\begin{theorem} \label{thm:relations_along_folding_chain}
For any orbit type $(H) \in \Phi_0(G)$ and pair of numbers $s_0,s_1 \in \bn$ with $s_1 \leq s_0$, the corresponding pair of folded orbit types $({}^{s_1}H),({}^{s_0}H) \in \Phi_0(G)$ have the relation $({}^{s_1}H) \leq ({}^{s_0}H)$ if and only if
\begin{align*}
     \mathfrak m(H) \mid \frac{s_0-s_1}{\gcd(s_0,s_1)} \text{ or } \mathfrak m(H) \mid \frac{s_0+s_1}{\gcd(s_0,s_1)} .
\end{align*}
\end{theorem}
\begin{proof}
Assume that a subgroup representative $H \in (H)$ has the amalgamated decomposition $K_O {}^{\varphi_O}\times^{\varphi_\Gamma}_{L} K_\Gamma$. According to the natural ordering of $\Phi_0(G)$, the relation 
\begin{align} \label{eq:natural_ordering_s_0s}
    ({}^{s_1}H) \leq ({}^{s_0}H)
\end{align}
is equivalent to the existence of an element $g \in G$ for which 
\begin{align} \label{rel:natural_ordering_equiv_s1s2}
    g\,{}^{s_1}Hg^{-1} \leq {}^{s_0}H.
\end{align}
Given any $s \in \bn$, put ${}^{s}K_O := \phi_s^{-1}(K_O)$, $\varphi^s_O := \varphi_O \circ \phi_s:{}^{s}K_O \rightarrow L$ and notice that the subgroup ${}^{s}H \leq G$ has the amalgamated decomposition
    \begin{align*}
        \psi_s^{-1}(K_O {}^{\varphi_O}\times^{\varphi_\Gamma}_{L} K_\Gamma) &= \psi_s^{-1}(\{ (x,y) \in K_O \times K_\Gamma \; : \; \varphi_O(x) = \varphi_\Gamma(y) \}) \\
        &= \{ (\phi_s^{-1}(x),y) \in \phi_s^{-1}(K_O) \times K_\Gamma \; : \; \varphi_O(x) = \varphi_\Gamma(y) \} \\
        &= \{ (x',y) \in {}^{s}K_O \times K_\Gamma \; : \; \varphi^s_O(x') = \varphi_\Gamma(y) \} \\
        &= {}^{s}K_O {}^{\varphi^s_O}\times^{\varphi_\Gamma}_{L} K_\Gamma.
    \end{align*}
On the other hand, and without loss of generality, take $g := (a,e_{\Gamma},1) \in O(2) \times \Gamma \times \mathbb{Z}_2$, put $\tilde{K}_O:= a K_O a^{-1} \leq O(2)$ and define the epimorphism $\tilde{\varphi}_O: \tilde{K}_O \rightarrow L$ by
\[
\tilde{\varphi}_O(x):= \varphi_O(a^{-1} x a), \quad x \in \tilde{K}_O,
\]
so that the subgroup $gHg^{-1} \leq G$ can be expressed, using amalgamated notation, as follows
    \begin{align*}
        gK_O {}^{\varphi_O}\times^{\varphi_\Gamma}_{L} K_\Gamma g^{-1} &= (a,e_\Gamma,1)\{ (x,y) \in K_O \times K_\Gamma \; : \; \varphi_O(x) = \varphi_\Gamma(y) \} (a^{-1},e_\Gamma,1) \\
        &= \{ (a x a^{-1}, y) \in a K_O a^{-1} \times  K_\Gamma \; : \; \varphi_O(x) = \varphi_\Gamma(y) \} \\
        &= \{ (\tilde{x},y) \in \tilde{K}_O \times K_\Gamma  \; : \; \tilde{\varphi}_O(\tilde{x}) = \varphi_\Gamma(y) \} \\
        &= \tilde{K}_O {}^{\tilde{\varphi}_O}\times^{\varphi_\Gamma}_{L} K_\Gamma.
    \end{align*}
With these notations, it becomes clear that the subgroups $ {}^{s_0}H, {}^{s_1}H \leq G$ have the relation
\eqref{rel:natural_ordering_equiv_s1s2} for some $g = (a,\gamma,\pm 1) \in O(2) \times \Gamma \times \bz_2$ if and only if $\varphi^{s_0}_O(axa^{-1}) = \varphi^{s_1}_O(x)$
for every $(x,y) \in {}^{\gcd(s_0,s_1)}K_O \times K_\Gamma$ satisfying $\varphi^{\gcd(s_0,s_1)}_O(x) = \varphi_\Gamma(y)$. In particular, since the $\Gamma$ component of an amalgamated subgroup is unchanged by folding, the relation \eqref{eq:natural_ordering_s_0s} is equivalent to the existence of an element $a \in O(2)$ satisfying
\[
\varphi^{s_1}_O(axa^{-1}) = \varphi^{s_0}_O(x) \text{ for all } x \in {}^{d}K_O, \text{ where } d:= \gcd(s_0,s_1).
\]
In turn, it is sufficient to only consider the behavior of the element $r^{\frac{1}{d}} \in {}^{d}K_O$ where $r$ is the rotation generator in $K_O$ and the two cases $a \in \{\kappa, e_{O(2)} \}$. If $a = \kappa$, then from $\kappa r \kappa = r^{-1}$, one has
\begin{align*}
   \kappa {}^{s_1}H \kappa \leq {}^{s_0}H \iff \varphi_O^{s_1}(\kappa r^{\frac{1}{d}} \kappa) = \varphi_O^{s_0}( r^{\frac{1}{d}} ) \iff \varphi_O(r)^{\frac{s_0+s_1}{d}} = e_L \iff \mathfrak m(H)  \mid \frac{s_0+s_1}{d}.
\end{align*}
Supposing instead that $a = e_{O(2)}$, one has
\begin{align*}
   {}^{s_1}H \leq {}^{s_0}H \iff \varphi_O^{s_1}(r^{\frac{1}{d}}) = \varphi_O^{s_0}( r^{\frac{1}{d}}) \iff \varphi_O(r)^{\frac{s_0-s_1}{d}} = e_L \iff \mathfrak m(H) \mid \frac{s_0-s_1}{d}.
\end{align*}
\end{proof}
\vs
For the sake of simplifying our exposition throughout the remainder of this paper, we employ the {\it Iverson brackets} which map any logical predicate $P$ to the set $\{0,1\}$ according to the rule
\begin{align} \label{notation_iverson}
    [P] := \begin{cases}
        1 \quad & \text{ if } P \text{ is true}; \\
        0 \quad & \text{ otherwise},
    \end{cases}
\end{align}
and for any set of indices $I \subset \bn$, we indicate by $\mathcal B_{\mathfrak m}(I)$ the Boolean expression
\begin{align}
\mathcal B_{\mathfrak m}(I) \equiv \text{ `` } \text{Every pair of indices } & x,y \in I \text{ satisfies either } \\ &\mathfrak m | \frac{(x-y)}{\gcd(x,y)} \text{ or } \mathfrak m | \frac{(x+y)}{\gcd(x,y)}. \text{"} \nonumber 
\end{align}
\begin{theorem}
Let $(H) \in \mathfrak M$ be an orbit type of maximal kind. For any pair of numbers $s_0,s_1 \in \bn$, one has 
\begin{align*}
\operatorname{coeff}^{{}^{s}H}(({}^{s_0}H) \cdot ({}^{s_1}H)) = \begin{cases}
    | W(H) | \quad & \text{ if } s = \gcd(s_0,s_1) \text{ and } [\mathcal B_{\mathfrak m(H)}(\{s_0,s_1\})]; \\
    0 \quad & \text{otherwise},
\end{cases}
\end{align*}
equivalently, for any $s \in \bn$, one has
\begin{align*}
\operatorname{coeff}^{{}^{s}H}(({}^{s_0}H) \cdot ({}^{s_1}H)) = |W(H)|[s = \gcd{(s_0,s_1)}] [\mathcal B_{\mathfrak m(H)}(\{s_0,s_1\})].    
\end{align*}
\end{theorem}
\begin{proof}
Recall, the integer coefficient $n_L \in \bz$ of an orbit type $(L)$ in the Burnside ring product $({}^{s_0}H) \cdot ({}^{s_1}H)$ is defined as the number of orbits of type $(L)$ contained in the $G$-space $G/{}^{s_0}H \times G/{}^{s_1}H$. In particular, $(L) \in \Phi_0(G)$ is an orbit type in $G/{}^{s_0}H \times G/{}^{s_1}H$ if and only if there exists a group element $g \in G$ such that
\begin{align*}
    L = {}^{s_0}H \cap g\,{}^{s_1}Hg^{-1}.
\end{align*}
Without loss of generality, put $g = (a,e_\Gamma, 1)$ with an arbitrary $a \in O(2)$ and notice that the amalgamated decomposition
\begin{align*}
    {}^{s_0}H \cap & g\,{}^{s_1}Hg^{-1} \\ &=  \left\{ (x,y) \in \left({}^{s_0}K_O \cap a {}^{s_1}K_O a^{-1}\right)  \times K_\Gamma: \varphi_O^{s_0}(x) = \varphi_O^{s_1}(a^{-1}xa) = \varphi_\Gamma(y)  \right\} \\
    &= \left({}^{s_0}K_O \cap {}^{s_1} \tilde K_O\right) {}^{\psi}\times_L^{\varphi_\Gamma} K_\Gamma, \quad (\text{where } {}^{s_1} \tilde K_O := a{}^{s_1}K_O a^{-1})
\end{align*}
describes a subgroup belonging to the conjugacy class $({}^{s}H)$ if and only if there exists a group element $b \in O(2)$ with both
\begin{enumerate}
    \item[(i)]  $\varphi_O(b x b^{-1}) = \varphi_O^{s_0}(x) = \varphi_O^{s_1}(a x a^{-1}); \text{ for all } x \in K_O,$
    \item[(ii)] $b {}^{s}K_O b^{-1} = {}^{s_0}K_O \cap a {}^{s_1}K_O a^{-1}.$
\end{enumerate}
The equivalence of condition $(\rm i)$ with the truthfulness of the Boolean expression $\mathcal B_{\mathfrak m(H)}(\{s,s_0,s_1\})$ has already been demonstrated in the proof of Theorem \ref{thm:relations_along_folding_chain}. On the other hand, the observation
\[
K_O \in \{O(2), SO(2), D_n, \bz_n, n \in \bn \},
\]
together with the facts 
\[
\bz_{s_0 n} \cap \bz_{s_1 n} = \bz_{\gcd(s_1,s_0) n} \text{ and } D_n = \bz_n \cup \kappa \bz_n,
\]
implies that condition $( \rm ii)$ can be equated with $s = \gcd(s_0,s_1)$
(cf. \cite{book-new}, Example 3.12). 
\vs
For convenience, put $d:= \gcd(s_0,s_1)$ and notice that, in the case one has $\mathcal B_{\mathfrak m(H)}(\{d,s_0,s_1\})$, the coefficient $n_{d} \in \bz$ standing next to the orbit type $({}^{d}H)$ in the Burnside ring product $({}^{s_0}H) \cdot ({}^{s_1}H) \in A(G)$ can be obtained with the recurrence formula (cf. Appendix \ref{sec:appendix}) as follows
\begin{align*}
 n_{d} = \frac{n({}^{d}H,{}^{s_0}H) \vert W({}^{s_0}H) \vert n({}^{d}H,{}^{s_1}H) \vert W({}^{s_1}H) \vert }{\vert W({}^{d}H) \vert},   
\end{align*}
and the result follows from the fact $\vert W(H) \vert = \vert W({}^{s}H) \vert$ for any $s \in \bn$ together with the observation that $\mathcal B_{\mathfrak m(H)}(\{d,s_0,s_1\})$ implies $n({}^{d}H,{}^{s_0}H) = n({}^{d}H,{}^{s_1}H) = 1$.
\end{proof}
\section{A Characterization of the Orbit 
Types of Maximal Kind in the Product of Basic Degrees.} \label{sec:char_maxorbtyps}
Since the classification put forward in the previous section fully characterizes the orbit types $(\tilde H)$ with $(\tilde H) > (H)$ for any orbit type of maximal kind $(H) \in \mathfrak M$, we are now in a position to calculate their coefficients in Burnside ring products of the form \eqref{eq:final_computation_gdegA}.
\vs
Given a basic degree $\deg_{\mathcal V_{m,j}} \in A(G)$ corresponding to an index pair $(m,j) \in \Sigma$ with $m >0$ and an orbit type of maximal kind $(H) \in \mathfrak M$, the recurrence formula for the $G$-equivariant Leray-Schauder degree (cf. Appendix \ref{sec:appendix}) implies
\begin{align} \label{rm5_f1}
    \deg_{\mathcal V_{m,j}} = (G) - y_{m,j}(H) + a_{m,j}, 
\end{align}
where $a_{m,j} \in A(G)$ is such that $\operatorname{coeff}^{{}^{s}H}(a_{m,j}) = 0$ for all $s \in \mathbb{N}$ and the coefficient $y_{m,j} \in \mathbb{Z}$ is determined by the rule
 \begin{align} \label{rm5_f2}
    y_{m,j} =
\begin{cases}
    0 \quad & \text{if } \dim{\mathcal V_{m,j}^{H}} \text{ is even};\\
    1 \quad & \text{if } \dim{\mathcal V_{m,j}^{H}} \text{ is odd and } \vert W(H) \vert = 2; \\
    2 \quad & \text{if } \dim{\mathcal V_{m,j}^{H}} \text{ is odd and } \vert W(H) \vert = 1,
\end{cases}
\end{align}
or, equivalently (see discussion of the Iverson bracket notation \eqref{notation_iverson}):
\begin{align*} 
%\label{rm5_f3}
    y_{m,j} = x_0\left[2 \nmid \dim{\mathcal V_{m,j}^{H}}\right],
\end{align*}
where
 \begin{align*} 
 %\label{rm5_f4}
    x_0 =
\begin{cases}
    1 \quad & \text{if } \vert W(H) \vert = 2; \\
    2 \quad & \text{if } \vert W(H) \vert = 1.
\end{cases}
\end{align*}
\begin{lemma}
Let $(H) \in \mathfrak M$ be an orbit type of maximal kind. For any basic degree $\deg_{\mathcal V_{m,j}}$ with $(m,j) \in \Sigma$ and $m > 0$ one has the implication
\[
\operatorname{coeff}^{H}(\deg_{\mathcal V_{m,j}}) \neq 0 \implies (H) \in \mathfrak M_{m,j}.
\]
\end{lemma}
\begin{proof}
Since $\mathcal V_{m,j}^{H} = \{0\}$ for {\it any} orbit type $(H) \notin \Phi_0(G, \mathscr H_{m,j})$, the proof follows from formula \eqref{rm5_f2} together with the assumption that $(H)$ is of maximal kind.
\end{proof}
\vs
For every $s \in \mathbb{N}$, the $s$-folding homomorphism \eqref{def:sfolding} induces the following relation between the basic degrees $\deg_{\mathcal V_{m,j}}, \deg_{\mathcal V_{sm,j}} \in A(G)$
\begin{align}\label{rm5_f5}
    \Psi_s(\deg_{\mathcal V_{m,j}}) = \deg_{\mathcal V_{sm,j}},
\end{align}
such that, for any orbit type $(H) \in \Phi_0(G)$, one has
\begin{align} \label{rm5_f6}
    \operatorname{coeff}^{H}(\deg_{\mathcal V_{m,j}}) = \operatorname{coeff}^{{}^{s}H}(\deg_{\mathcal V_{sm,j}}).
\end{align}
Using the observations \eqref{rm5_f1}, \eqref{rm5_f5}, and \eqref{rm5_f6}, we will derive results which characterize the behavior of orbit types of maximal kind in the Burnside ring product of a finite collection of basic degrees of the form
\begin{align} \label{def:finite_collection_bdeg}
    \left\{ \deg_{\mathcal V_{s_km, j_k}} \in A(G)  :  s_k \in \mathbb{N} \right\}_{k=1,2,\ldots,N}, \quad m > 0, \; j_k \in \{0,1,\ldots,r\}.
\end{align}
To begin with, we recall a result that the coefficient of an orbit type of maximal kind $(H) \in \mathfrak M$ is \textit{$2$-nilpotent} with respect to the Burnside ring product of any pair of basic degrees $\deg_{\mathcal V_{m,j_1}}, \deg_{\mathcal V_{m,j_2}} \in A(G)$ with $m > 0$ and $j_1,j_2 \in \{0,1,\ldots,r\}$, provided that both $\dim\mathcal V_{m,j_1}^{H}$ and $\dim\mathcal V_{m,j_2}^{H}$ are odd. Although this fact has been demonstrated elsewhere (cf. for example \cite{book-new}, \cite{AED}, \cite{BalChen}, \cite{Carlos}), the method of its proof is instructive for the proofs of subsequent, novel results and for this reason we include our own version of the argument.
\begin{lemma} \label{lemm:product_2bdegs_fixed_folding}
Let $(H) \in \mathfrak M$ be an orbit type of maximal kind. For any $m > 0$ and $j_1,j_2 \in \{0,1,\ldots,r\}$, one has 
\begin{align*}
    \operatorname{coeff}^{H}(\deg_{\mathcal V_{m,j_1}} \cdot \deg_{\mathcal V_{m,j_2}} ) =
        \begin{cases}
                0 \quad & \text{if } \dim\mathcal V_{m,j_1}^{H} \text{ and } \dim\mathcal V_{m,j_2}^{H} \\ & \quad \text{ are of the same parity;} \\
                - x_0 \quad & \text{else,}
            \end{cases}
    \end{align*}
    or, equivalently
    \begin{align*}
    \operatorname{coeff}^{H}(\deg_{\mathcal V_{m,j_1}} \cdot \deg_{\mathcal V_{m,j_2}} ) = -x_0\left[2 \nmid \dim\mathcal V_{m,j_1}^{H} + \dim\mathcal V_{m,j_2}^{H}\right].
    \end{align*}
\end{lemma}
\begin{proof}
Consider the Burnside ring product of the relevant basic degrees
\begin{align*}
    \deg_{\mathcal V_{m,j_1}} \cdot \deg_{\mathcal V_{m,j_2}} & = \left( (G) - y_{m,j_1}(H) + a_{m,j_1} \right) \cdot \left( (G) - y_{m,j_2}(H) + a_{m,j_2} \right)\\
    &= (G) - (y_{m,j_1} + y_{m,j_2} - y_{m,j_1}y_{m,j_2}\vert W(H)\vert) (H) + a, 
\end{align*}
where $a \in A(G)$ is such that $\operatorname{coeff}^{{}^{s}H}(a) = 0$ for all $s \in \mathbb{N}$. Now, if $\dim\mathcal V_{m,j_1}^{H}$ and $\dim\mathcal V_{m,j_2}^{H}$ are both even, then one has $y_{m,j_1} = y_{m,j_2} = 0$ and the result follows. On the other hand, if $\dim\mathcal V_{m,j_1}^{H}$ and $\dim\mathcal V_{m,j_2}^{H}$ are both odd, then one has $y_{m,j_1} = y_{m,j_2} = x_0$ such that
\begin{align*}
    y_{m,j_1} + y_{m,j_2} - y_{m,j_1}y_{m,j_2}\vert W(H)\vert = x_0(2-x_0 \vert W(H) \vert),
\end{align*}
where in either of the cases: $x_0 = 2$ and $\vert W(H) \vert = 1$ or $x_0 = 1$ and $\vert W(H) \vert = 2$, one has $2 - x_0 \vert W(H) \vert = 0$. Suppose instead that $\dim\mathcal V_{m,j_1}^{H}$ and $\dim\mathcal V_{m,j_2}^{H}$ are of different parties and one has the two corresponding cases, $y_{m,j_1} = x_0$ and $y_{m,j_2} = 0$ \textit{or} $y_{m,j_1} = 0$ and $y_{m,j_2} = x_0$, both of which imply $y_{m,j_1} + y_{m,j_2} - y_{m,j_1}y_{m,j_2}\vert W(H)\vert = x_0$.
\end{proof}
\vs
Since our computations in the Burnside ring often involve the product of a number of basic degrees, we consider a natural generalization of Lemma \eqref{lemm:product_2bdegs_fixed_folding} to an orbit type of maximal kind $(H) \in \mathfrak M$ and the Burnside ring product of any finite collection of basic degrees $\{\deg_{\mathcal V_{m,j_k}} \}_{k=1}^N$ with $m > 0$ and $j_1,\ldots,j_N \in \{0,1,\ldots,r\}$.
\begin{corollary} \label{cor:product_Nbdegs_fixed_folding}
Let $(H) \in \mathfrak M$ be an orbit type of maximal kind. For any finite collection of basic degrees $\{ \deg_{\mathcal V_{m,j_k}} \}_{k=1}^N$ with $m > 0$ and $j_1,\ldots,j_N \in \{0,1,\ldots,r\}$, one has
\begin{align*} 
%\label{l3_res}
    \operatorname{coeff}^{H} \left( \prod\limits_{k=1}^N \deg_{\mathcal V_{m,j_k}} \right) =
        \begin{cases}
                0 \quad & \text{if } \dim\mathcal V_{m,j_k}^{H} \text{ is odd } \\ & \quad \text{for an even number of } j_k \in \{0,1,\ldots,r\};\\
                - x_0 \quad & \text{otherwise,}
            \end{cases}
    \end{align*}
    or, equivalently
    \begin{align*} 
    \operatorname{coeff}^{H} \left( \prod\limits_{k=1}^N \deg_{\mathcal V_{m,j_k}} \right) = - x_0\left[2 \nmid \sum\limits_{k=1}^N \dim\mathcal V_{m,j_k}^{H}\right].
    \end{align*}
\end{corollary}
\vs
The previous result only concerns the product of basic degrees associated with a fixed Fourier mode. We are also interested in characterizing the behaviour of the coefficients of an orbit type 
of maximal kind $(H) \in \mathfrak M$ and its $s$-foldings $\{ ({}^{s}H) \}_{s \in \bn} \subset \mathfrak M$ in the Burnside ring product of any pair of basic degrees $\deg_{\mathcal V_{s_1m,j_1}}, \deg_{\mathcal V_{s_2m,j_2}} \in A(G)$ with $m > 0$, $j_1,j_2 \in \{0,1,\ldots,r\}$ and where $s_1,s_2 \in \bn$ are any two numbers satisfying $s_1 \neq s_2$.
\begin{lemma} \label{lemm:product_2bdegs}
Let $(H) \in \mathfrak M$ be an orbit type of maximal kind. For any pair of basic degrees $\deg_{\mathcal V_{s_1m,j_1}}, \deg_{\mathcal V_{s_2m,j_2}} \in A(G)$ with $m > 0$, $j_1,j_2 \in \{0,1,\ldots,r\}$ and $s_1,s_2 \in \bn$ where $s_1 \neq s_2$ and for any number $s_0 \in \bn$, one has
\begin{align*}
\operatorname{coeff}^{{}^{s_0}H}&(\deg_{\mathcal V_{s_1m,j_1}} \cdot \deg_{\mathcal V_{s_2m,j_2}}) = -x_0\left[2 \nmid \dim\mathcal V_{m,j_0}^{H}\right]\left[s_0 \in \{s_1,s_2\}\right] \\ & \quad + 2x_0\left[2 \nmid \dim\mathcal V_{m,j_1}^{H} \dim\mathcal V_{m,j_2}^{H}\right]\left[\operatorname{gcd}(s_1,s_2) = s_0\right][\mathcal B_{\mathfrak m(H)}(\{s_0,s_1,s_2\})], \nonumber
\end{align*}
where \begin{align*}
    j_0 :=
    \begin{cases}
        j_1 & \text{ if } s_0=s_1; \\
        j_2 & \text{ if } s_0 = s_2.
    \end{cases}
\end{align*}
\end{lemma}
\begin{proof}
    Consider the Burnside ring product of the relevant basic degrees
    \begin{align} \label{eq:2bdeg_product}
    \deg_{\mathcal V_{s_1m,j_1}} \cdot &\deg_{\mathcal V_{s_2m,j_2}} = \left( (G) - y_{m,j_1}({}^{s_1}H) + a_{m,j_1} \right) \cdot \left( (G) - y_{m,j_2}({}^{s_2}H) + a_{m,j_2}\right) \nonumber \\
        &= (G) - y_{m,j_1}({}^{s_1}H) - y_{m,j_2}({}^{s_2}H) + y_{m,j_1}y_{m,j_2}({}^{s_1}H) \cdot ({}^{s_2}H) + a \nonumber \\
        &= (G) - y_{m,j_1}({}^{s_1}H) - y_{m,j_2}({}^{s_2}H) \\ & \quad \quad \quad + y_{m,j_1}y_{m,j_2}\vert W(H) \vert ({}^{d}H)
    [\mathcal B_{\mathfrak m(H)}(\{d ,s_1,s_2\})] + a \nonumber
    \end{align}
where $d := \operatorname{gcd}(s_1,s_2)$ and $a \in A(G)$ satisfies $\operatorname{coeff}^{{}^{s}H}(a) = 0$ for all $s \in \mathbb{N}$. In order to demonstrate each of the boolean conditions, we consider the following cases: 
\begin{enumerate}
    \item[$(i)$] Supposing that $s_0 \notin \{s_1,s_2 \}$, if it is also the case that both $\dim\mathcal V_{m,j_1}^{H}$ and $ \dim\mathcal V_{m,j_2}^{H}$ are odd \textit{and} $\operatorname{gcd}(s_1,s_2) = s_0$, \textit{and} $\mathcal B_{\mathfrak m(H)}(\{s_0 ,s_1,s_2\})$  then one has
\begin{align*}
\operatorname{coeff}^{{}^{s_0}H}(\deg_{\mathcal V_{s_1m,j_1}} \cdot \deg_{\mathcal V_{s_2m,j_2}}) = x_0^2 \vert W(H) \vert ({}^{s_0}H) 
\end{align*}
and the result follows from the fact that $x_0 \vert W(H) \vert = 2$. On the other hand, if \textit{either} of $\dim\mathcal V_{m,j_1}^{H}$ and $\dim\mathcal V_{m,j_2}^{H}$ are even \textit{or} if $\operatorname{gcd}(s_1,s_2) \neq s_0$, \textit{ or } if 
$\mathcal B_{\mathfrak m(H)}(\{s_0 ,s_1,s_2\})$ does not hold,
then the result follows immediately from the assumption that $s_0 \notin \{s_1,s_2 \}$.
\item[$(ii)$] Supposing instead that $s_0 \in \{s_1,s_2\}$, if it is also the case that $\dim\mathcal V_{m,j_0}^{H}$ is even, then one has $y_{m,j_0} = 0$ and the result follows from \eqref{eq:2bdeg_product}. If one has $s_0 \in \{s_0,s_1\}$, with both $\dim\mathcal V_{m,j_1}^{H}$, $\dim\mathcal V_{m,j_{s_2}}^{H}$ odd {\it and} $s_0 = \gcd(s_1,s_2)$ {\it and} $\mathcal B_{\mathfrak m(H)}(\{s_0 ,s_1,s_2\})$, then \eqref{eq:2bdeg_product} becomes
\begin{align*}
    \deg_{\mathcal V_{s_1m,j_1}} \cdot \deg_{\mathcal V_{s_2m,j_2}} &= (G) - x_0({}^{s_0}H) + 2x_0({}^{s_0}H) + a',
\end{align*}
where $a' \in A(G)$ is such that $\operatorname{coeff}^{{}^{sm}H}(a') = 0$, and the result follows. On the other hand, if either of $\dim\mathcal V_{s_1m,j_{s_1}}^{H}, \dim\mathcal V_{s_1m,j_{s_1}}^{H}$ are even (in which case one has $y_{m,j_1}y_{m,j_2} = 0$) \text{or} if $\operatorname{gcd}(t,l) \neq s_0$, \text{or} if $\mathcal B_{\mathfrak m(H)}(\{s_0 ,s_1,s_2\})$ does not hold, then the result follows from \eqref{eq:2bdeg_product}.
\end{enumerate}
\end{proof}
\vs
What follows is the final generalization of the previous results to an orbit type of maximal kind $(H) \in \mathfrak M$ and its $s$-foldings $\{ ({}^{s}H) \}_{s \in \bn} \subset \mathfrak M$ in the Burnside ring product of any finite collection of basic degrees $\{\deg_{\mathcal V_{s_km,j_{s_k}}} \}_{k=1}^N$ with $m > 0$, $j_{s_1},\ldots,j_{s_N} \in \{0,1,\ldots,r\}$ and where
$s_1,\ldots,s_{N} \in \mathbb{N}$ are distinct.
\begin{theorem}\label{thm:product_Nbdegs}
Let $(H) \in \mathfrak M$ be an orbit type of maximal kind. For any finite collection of basic degrees $\{\deg_{\mathcal V_{s_km,j_{s_k}}} \}_{k=1}^N$ with $m > 0$, $j_{s_1},\ldots,j_{s_N} \in \{0,1,\ldots,r\}$ and where $s_1,\ldots,s_{N} \in \mathbb{N}$ are distinct and for any number $s_0 \in \bn$, one has
\begin{align*} 
%\label{l3_res}
& \operatorname{coeff}^{{}^{s_0}H}  \left( \prod\limits_{k=1}^N \deg_{\mathcal V_{s_km,j_{s_k}}} \right)  =  -x_0 \left[ 2 \nmid \dim \mathcal V_{m,j_{s_0}}^{H} \right] \left[s_0 \in \{s_1,\ldots,s_N\}\right] \\ & \quad \quad + 2x_0 \sum\limits_{\substack{ I \subset \{s_1,s_2,\ldots,s_N\} \\ I \neq \emptyset, \{s_0\} }} (-2)^{\vert I \vert-2} 
\left[ \mathcal B_{\mathfrak m(H)}(I) \right] \left[s_0 = \operatorname{gcd}(I)  \right] \left[2 \nmid \prod_{s \in I} \dim \mathcal V_{m,j_s}^{H} \right], \nonumber
    \end{align*}
where, if $s_0 = s_k$ for any $k = 1,\ldots,N$, then $j_{s_0} := j_{s_k}$. 
\end{theorem}
\begin{proof}
Consider the Burnside ring product of the relevant basic degrees
\begin{align} \label{prod_1}
    \prod\limits_{k=1}^N \deg_{\mathcal V_{s_km,j_{s_k}}} &= \prod\limits_{k=1}^N  (G) - y_{m,j_{s_k}}({}^{s_k}H) + b_{m,j_{s_k}},
\end{align}
and also the related Burnside ring product
\begin{align} \label{prod_2}
    \prod\limits_{k=1}^N  (G) - y_{m,j_{s_k}}({}^{s_k}H).
\end{align}
Notice that for any $s_0 \in \mathbb{N}$, one has
\begin{align*}
%\label{eq:first_computation_gdegA}
\operatorname{coeff}^{{}^{s_0}H}\left(\prod\limits_{k=1}^N \deg_{\mathcal V_{s_km,j_{s_k}}} \right) = \operatorname{coeff}^{{}^{s_0}H}\left( \prod\limits_{k=1}^N  (G) - y_{m,j_{s_k}}({}^{s_k}H)  \right),
\end{align*}
and also that \eqref{prod_2} has the expansion 
\begin{align*} 
    \prod\limits_{k=1}^N  (G) - &y_{m,j_{s_k}}({}^{s_k}H) = \sum\limits_{I \subset \{s_1,s_2,\ldots,s_N\}} \prod\limits_{s \in I} -y_{m,j_s} ({}^{s}H) \nonumber \\
    &= \sum\limits_{I \subset \{s_1,s_2,\ldots,s_N\}} \vert W(H) \vert^{\vert I \vert -1} \left(\prod\limits_{s \in I} -y_{m,j_s} \right) \left[ \mathcal B_{\mathfrak m(H)}(I) \right] ({}^{\operatorname{gcd}(I)}H), 
\end{align*}
where the notation $I$ is used to indicate a subset of the indices $\{s_1,s_2,\ldots, s_N \}$ and the expression $\sum_{I \subset \{s_1,s_2,\ldots, s_N \}}$ describes a summation over all such subsets, including the empty set, in which case we put $\prod_{s \in \emptyset} -y_{m,j_s}(\alpha_*) ({}^{s}H) := (G)$, and the full set. It follows that the coefficient of $({}^{s_0}H) \in \Phi_0(G; \mathscr H)$ in the Burnside ring product \eqref{prod_1}, which for convenience of notation we denote by
\[
n_{s_0} := \operatorname{coeff}^{{}^{s_0}H}(\prod\limits_{k=1}^N \deg_{\mathcal V_{s_km,j_{s_k}}}),
\]
is specified by the formula
\begin{align} \label{coeff_1}
n_{s_0} = \sum\limits_{\substack{ I \subset \{s_1,s_2,\ldots,s_N\} \\ I \neq \emptyset}} \vert W(H) \vert^{\vert I \vert -1} \left[ \mathcal B_{\mathfrak m(H)}(I) \right] [s_0 = \operatorname{gcd}(I) ] \prod\limits_{s \in I} -y_{m,s}. 
\end{align}
Now, for any non-empty subset of indices $I \subset \{s_1,s_2,\ldots,s_N\}$, the product 
\begin{align}
\prod\limits_{s \in I} -y_{m,j_s} = \prod\limits_{s \in I} -x_0\left[2 \nmid \dim \mathcal V_{m,j_s}^{H}\right],
\end{align}
is determined by the rule
\begin{align*}
    \prod\limits_{s \in I} -y_{m,j_s} = 
        \begin{cases}
        (-x_0)^{\vert I \vert} \quad& \text{ if } \dim \mathcal V_{m,j_s}^{H} \text{ is odd for all } s \in I;\\
        0 \quad& \text{ otherwise}.
    \end{cases} 
\end{align*}
Therefore, \eqref{coeff_1} becomes
\begin{align*}
%\label{coeff_2}
&\sum\limits_{\substack{ I \subset \{s_1,s_2,\ldots,s_N\} \\ I \neq \emptyset}} \vert W(H) \vert^{\vert I \vert -1}
(-x_0)^{\vert I \vert}  \left[ \mathcal B_{\mathfrak m(H)}(I) \right] \left[s_0 = \operatorname{gcd}(I)  \right] \left[ 2 \nmid \prod_{s \in I} \dim \mathcal V_{m,j_s}^{H} \right]  \nonumber \\
& =  -x_0 \left[ 2 \nmid \dim \mathcal V_{m,j_0}^{H} \right] \left[s_0 \in \{s_1,\ldots,s_N\}\right] \nonumber \\ & \quad + 2x_0 \sum\limits_{\substack{ I \subset \{s_1,s_2,\ldots,s_N\} \\ I \neq \emptyset, \{s_0\} }} (-2)^{\vert I \vert-2} \left[ \mathcal B_{\mathfrak m(H)}(I) \right] \left[s_0 = \operatorname{gcd}(I)  \right] \left[2 \nmid \prod_{s \in I} \dim \mathcal V_{m,j_s}^{H} \right],
\end{align*}
where we have used the fact that $x_0 \vert W(H) \vert = 2$.
\end{proof}
\vs
The following practical Corollary will prove useful in subsequent sections.
\begin{corollary}\label{cor:product_Nbdegs}
    Let $(H) \in \mathfrak M$ be an orbit type of maximal kind. For any finite collection of basic degrees $\{\deg_{\mathcal V_{s_km,j_{s_k}}} \}_{k=1}^N$ with $m > 0$, $j_{s_1},\ldots,j_{s_N} \in \{0,1,\ldots,r\}$ and where $s_1,\ldots,s_{N} \in \mathbb{N}$ are distinct and for any $s \in \{s_1,\ldots,s_{N}\}$ with $2 \nmid \dim \mathcal V_{m,j_{s}}^H$, 
\begin{align}
\operatorname{coeff}^{{}^{s}H} & \left( \prod\limits_{k=1}^N \deg_{\mathcal V_{s_km,j_{s_k}}} \right)  \neq 0.
    \end{align}
\end{corollary}
\begin{proof}
    Indeed, notice that $-x_0 + 2x_0 C \neq 0$ for any value of $C \in \bn$.
\end{proof}

\section{On the Non-Triviality of the Degree $\gdeg(\mathscr A, B(\mathscr H))$}\label{sec:nontriviality_deg}
Given an orbit type of maximal kind $(H) \in \mathfrak M$ and a positive number $s > 0$, we put 
\begin{align*}
\begin{cases}
  \Sigma^s(H) := \{ (m,j) \in \Sigma_0 : 2 \nmid \dim \mathcal V_{m,j}^{{}^{s}H} \}; \\   
\mathfrak n^s(H) := | \Sigma^s(H) |,  
\end{cases}
\end{align*}
and
\begin{align*} 
%\label{def:set_S(H)}
 S(H):= \{ s \in \bn : \mathfrak n^s(H) \neq 0 \}.  
\end{align*}
\begin{remark}
If $(H) \in \mathfrak M_{m_0}$ for some $m_0 >0$, then $(m,j) \in \Sigma^s(H)$ only if $m = m_0s$.      
\end{remark}
We are now in a position to present our main result. 
\begin{proposition} \label{prop:main_existence}
Let $(H) \in \mathfrak M$ be an orbit type of maximal kind. For any $s_0 \in \bn$, one has 
\begin{align*} 
\operatorname{coeff}^{{}^{s_0}H} & \left( \gdeg\left(\mathscr A,B(\mathscr H) \right) \right)  =  -x_0 \left[ 2 \nmid \mathfrak n^{s_0}(H) \right] \left[s_0 \in  S(H) \right] \\ & \quad + 2x_0 \sum\limits_{\substack{ I \subset S(H) \\ I \neq \emptyset, \{s_0\} }} (-2)^{\vert I \vert-2} 
\left[ \mathcal B_{\mathfrak m(H)}(I) \right] \left[s_0 = \operatorname{gcd}(I)  \right] \left[2 \nmid \prod_{s \in I} \mathfrak n^{s}(H) \right]. \nonumber
    \end{align*}
\end{proposition}
\begin{proof}
For convenience of notation, put
\[
\Sigma^0(H) := \{ (m,j) \in \Sigma_0: 2 \mid \dim \mathcal V_{m,j}^{{}^{s}H} \text{ for all } s > 0 \},
\]
and define
\begin{align*}
%\label{def:rho_s_G}
 \rho^s_G(H) := \prod_{(m,j) \in \Sigma^s(H)} \deg_{\mathcal V_{m,j}}. 
\end{align*}
Enumerating the set $S(H) = \{s_1,s_2,\ldots, s_N \}$, 
one clearly has
\[
\Sigma_0  \setminus \Sigma^0(H) = \bigcup_{k=1}^N \Sigma^{s_k}(H),
\]
such that the degree computation \eqref{eq:final_computation_gdegA} becomes
\[
\gdeg(\mathscr A, B(\mathscr H)) = \rho^0_G(H) \cdot \prod_{k  = 1}^{N} \rho^{s_k}_G(H).
\]
From Lemma \ref{lemm:product_2bdegs_fixed_folding} and Corollary \ref{cor:product_Nbdegs_fixed_folding}, one has
\begin{align*}
\rho_G^s(H) = \begin{cases}
    (G) + b_0 & \text{ if } s = 0; \\
    (G) - y_s({}^{s}H) + b_s & \text{ if } s > 0,
\end{cases}    
\end{align*}
where $b_0,b_s$ satisfy $\operatorname{coeff}^{{}^{s}H}(b_0) = \operatorname{coeff}^{{}^{s}H}(b_s) = 0$ for all $s \in \bn$ and
\begin{align*} 
%\label{rule:coefficient_ys}
y_s = \begin{cases}
    0 & \text{ if } \mathfrak n^s(H) \text{ is even}; \\
    x_0 & \text{ if } \mathfrak n^s(H) \text{ is odd}.
\end{cases}    
\end{align*}
Assume, without loss of generality, that $(H) \in \mathfrak M_{m_0}$ and notice that each of the elements $\rho_G^{s_k}(H)$ has the form of a basic degree corresponding to an index $(s_k m_0, j_{s_k}) \in \Sigma$ satisfying
\begin{align*}
2 \mid \dim \mathcal V_{m_0,j_{s_k}}^H + \mathfrak n^{s_k}(H),
\end{align*}
i.e., such that 
\[
\operatorname{coeff}^{{}^{s_k}H}\left(\deg_{\mathcal V_{s_k m_0, j_{s_k}}} \right) = \operatorname{coeff}^{{}^{s_k}H}\left(\rho_G^{s_k}(H)\right).
\]
Therefore, the coefficient of any $s_0$-folding $({}^{s_0}H) \in \mathfrak M$ in the degree \eqref{eq:final_computation_gdegA} can be computed as follows
\[
\operatorname{coeff}^{{}^{s_0}H}(\gdeg(\mathscr A, B(\mathscr H))) = \operatorname{coeff}^{{}^{s_0}H}\left(\prod_{k=1}^N \deg_{\mathcal V_{s_k m_0, j_{s_k}}}\right),
\]
and the result follows from Theorem \ref{thm:product_Nbdegs}.
\end{proof}
\vs
Just as Corollary \ref{cor:product_Nbdegs} naturally follows from Theorem \ref{thm:product_Nbdegs}, the following Corollary is a direct implication of Proposition \ref{prop:main_existence}.
\begin{corollary} \label{cor:main_corollary}
    Let $(H) \in \mathfrak M$ be an orbit type of maximal kind. For any $s \in \bn$ with $2 \nmid n^{s}(H)$, one has
    \[
    \operatorname{coeff}^{{}^{s}H} \left( \gdeg\left(\mathscr A,B(\mathscr H) \right) \right) \neq 0.
    \]
\end{corollary}
In turn, combining Lemma \ref{lemm:sufficient_condition} with Corollary \ref{cor:main_corollary} leads immediately to Theorem \ref{thm:main_theorem}.
\section{Motivating Example: An Arrangement of $N$ Coupled Pendula with Dihedral Symmetries and Subject to Nonlinear Forcing} \label{sec:n_pendula}
\begin{figure}
    \centering
\includegraphics[width=0.5\linewidth]{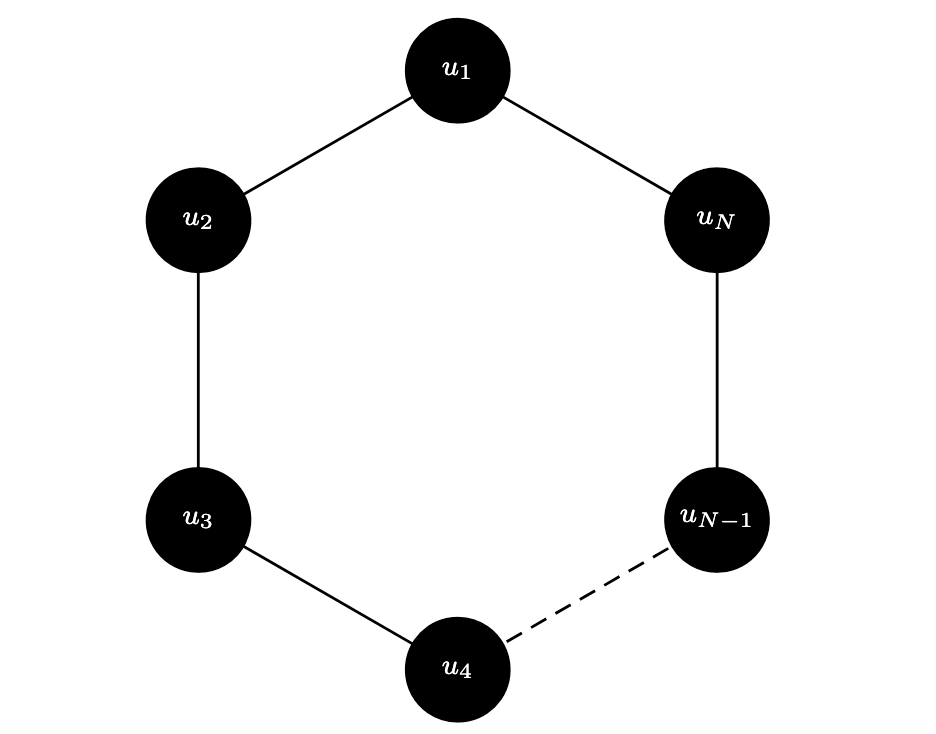}
    \caption{Cycle of $N$ oscillating pendula with $\Gamma=D_{N}$ symmetries}
    \label{fig:dihedral_graph}
\end{figure}
We can provide a more definitive prediction of the symmetries expressed by non-stationary solutions to \eqref{eq:system_periodic} than that established in Theorem \ref{thm:main_theorem} by prescribing an exact coupling relation between the pendula in our configuration. Specifically, let $\Gamma$ be the dihedral group of order $2N$ such that our full symmetry group becomes
\[
G := O(2) \times D_N \times \bz_2.
\]
Our model for this configuration is the system
\begin{align} \label{eq:system_example}
   \begin{cases} 
        \ddot u(t) = |u|^qu - (L + \id)u, \quad t \in \br, \;  u(t) \in V \\
        u(t) = u(t+ p), \; \dot u(t) = \dot u(t+p),
    \end{cases}
\end{align}
where $|u|^qu := (|u|^qu_1, |u|^qu_2, \ldots, |u|^qu_N)$ for any even $q > 1$ and $L:V \rightarrow V$ is the {\it graph Laplacian matrix} for an  undirected graph $\mathbb G$ invariant under the permutation action of $D_N$, with $N$ vertices representing a collection of the same number of oscillating pendula and with edges representing the coupling relations between vertices.
\vs
Since the natural permutation representation of $\Gamma$ on $V$
\begin{align} 
\rho_V: \Gamma \rightarrow GL(V), \quad \rho_V(\sigma)(u_1,u_2,\ldots,u_N) := (u_{\sigma(1)},u_{\sigma(2)},\ldots,u_{\sigma(N)}),
\end{align}
is induced by the permutation actions of the rotation and reflection
generators $\gamma, \kappa \in \Gamma$ on the indices $i \in \{1,\ldots,N\}$ 
\[
\gamma(i):= i+1 \pmod{N} \text{ and } \kappa(i) := N-i \pmod{N},
\]
its character can be determined according to the rule
\begin{align*}
    \chi_V(\gamma) = 0 \; \text{  and  } \; \chi_V(\kappa) = 
    \begin{cases}
        1 & \text{ if } N \text{ is  odd}; \\
        0 & \text{ if } N  \text{ is  even}.
    \end{cases}
\end{align*}
The number and character of the irreducible $\Gamma$ representations also depend on the dihedral order $N$: there are always the {\it trivial representation} $\mathcal V_0 \simeq \br$, on which $\Gamma$ acts trivially, $\lfloor \frac{N+1}{2} \rfloor - 1$ {\it geometric representations} $\mathcal V_j$, each with an action induced by the corresponding matrix representation $\rho_j: \Gamma \rightarrow GL(\bc)$
    \[
\rho_j(\kappa) := \begin{pmatrix}
    1&0 \\
    0&-1
\end{pmatrix}, \quad \rho_j(\gamma):=\begin{pmatrix}
    \cos(\frac{2j\pi}{N}) & -\sin(\frac{2j\pi}{N}) \\
    \sin(\frac{2j\pi}{N}) & \cos(\frac{2j\pi}{N})
\end{pmatrix}, \quad 1 \leq j < \lfloor \frac{N+1}{2} \rfloor,
\]
the {\it sign representation} $\mathcal V_{*} \simeq \br$, with the action
\[
\sigma \cdot v := \sign(\sigma) v, \quad v \in \mathcal V_{*},
\]
and, in the case that $N$ is even, two additional irreducible one-dimensional representations $\mathcal V_{\lfloor \frac{N+1}{2} \rfloor}, \mathcal V_{**} \simeq \br$ equipped, respectively, with the actions
\[
\sigma \cdot v := - \sign(\sigma) v, \quad v \in \mathcal V_{\lfloor \frac{N+1}{2} \rfloor},
\]
and
\[
\rho_{**}(\kappa) = \rho_{**}(\gamma) = -1.
\]
Comparing characters for the irreducible representations of $\Gamma$ with the character of $V$
\begin{table}[h]
\centering
\begin{tabular}{|c|ccccc|}
\hline
conjugacy classes &$e_\Gamma$ & \; & $\kappa$ & \; &$\gamma$\\\hline 
$\chi_0$ &$1$ & \; & $1$ & \; & $1$ \\
$\chi_1$ &$2$ & \; & $0$ & \; & $\cos(\frac{2 \pi}{N})+\cos(\frac{2 \pi}{N})$\\
$\vdots$ & $\vdots$ & \;& $\vdots$ & \; & $\vdots$ \\
$\chi_j$ &$2$ & \; & $0$ & \; & $\cos(\frac{2 j \pi}{N})+\cos(\frac{2 j\pi}{N})$\\
$\vdots$ & $\vdots$ & \;& $\vdots$ & \; & $\vdots$ \\
$\chi_{\lfloor \frac{N+1}{2} \rfloor}$ &$1$ & \; & $-1$ & \; & $1$\\
$\chi_{*}$ &$1$ & \; & $1$ & \; & $-1$\\
$\chi_{**}$ &$1$ & \; & $-1$ & \; & $-1$\\
\hline 
$\chi_V$ &$N$ & \; & $\frac{1}{2}(1-(-1)^N)$ & \; & $0$ \\
\hline 
\end{tabular}
\vs
\caption{Character Table for $D_N$.}
\end{table} \\
one obtains the relation
\begin{align}
    \chi_V = 
    \begin{cases}
        \chi_0 + \chi_1 + \cdots + \chi_{\frac{N+1}{2} - 1} &  \text{if } 2 \nmid N; \\
        \chi_0 + \chi_1 + \cdots + \chi_{\frac{N}{2}-1} + \chi_{\frac{N}{2}} & \text{if } 2 \mid N,
    \end{cases}
\end{align}
implying that $V$ has the $\Gamma$-isotypic decomposition 
\begin{align}
    V = \begin{cases}
        V_0 \oplus V_1 \oplus \cdots \oplus V_{\frac{N+1}{2} - 1} & \text{if } 2 \nmid N; \\
       V_0 \oplus V_1 \oplus \cdots \oplus V_{\frac{N}{2}-1} \oplus V_{\frac{N}{2}} & \text{if } 2 \mid N.
    \end{cases}
\end{align}
For the sake of generality, we adopt the following notation for the set of $\Gamma$-isotypic indices relevant to the $\Gamma$-isotypic decomposition of $V$
\[
\mathfrak J(N) := \begin{cases}
    \{0,1,\ldots, \frac{N+1}{2} - 1\} &  \text{if } 2 \nmid N; \\
    \{0,1,\ldots, \frac{N}{2} - 1, \frac{N}{2}\} &  \text{if } 2 \mid N.
\end{cases}
\]
The graph Laplacian associated with the undirected graph $\mathbb G$ with $N$ vertices invariant under the permutation action of $D_N$ has the form
\[
L=\left(
\begin{array}
[c]{ccccccc}%
-2 & 1 & 0 & \dots & 0 & 0 & 1\\
1 & -2 & 1 & \dots & 0 & 0 & 0\\
0 & 1 & -2 & \dots & 0 & 0 & 0\\
\vdots & \vdots & \vdots & \ddots & \vdots & \vdots & \vdots\\
0 & 0 & 0 & \dots & -2 & 1 & 0\\
0 & 0 & 0 & \dots & 1 & -2 & 1\\
1 & 0 & 0 & \dots & 0 & 1 & -2
\end{array}
\right).
\]
One can verify that such a matrix has the eigenvectors
\[
v_j := (1,\gamma^j,\gamma^{2j},\ldots,\gamma^{(N-1)j})^T, \quad j \in \{0,1,\ldots,\lfloor \frac{N+1}{2} \rfloor - 1\}, \; \gamma := e^{\frac{2\pi i}{N}}
\]
corresponding to the eigenvalues
\[
-z_j := - 2 + \gamma^j + \gamma^{-j} = -4 \sin^2 \left( \frac{\pi j}{N} \right).
\]
And, in the case that $N$ is even, one will find that there is also the eigenpair
\[
v_{\frac{N}{2}} := (1,-1,1\ldots, -1)^T, \quad z_{\frac{N}{2}} = -4.
\]
Since each isotypic component of $V$ has simple isotypic multiplicity, i.e. since
\[
m_j = 1, \quad j \in \mathfrak J(N),
\]
the eigenvalues of $\mathscr A$ (cf. Definition \eqref{def:operator_eigenvalues}) have the form
\[
\mu_{m,j} = \frac{m^2 - \beta^2(z_j + 1)}{1 + m^2}.
\]
Therefore, the condition $\mu_{m,j} < 0$ is equivalent to the condition
\[
4 \sin^2 \left( \frac{\pi j}{N} \right) > \frac{m^2}{\beta^2} + 1
\]
and Theorem \ref{thm:main_theorem} becomes:
\begin{theorem}\label{thm:N_Pendulum}
Take $m > 0$ and let $(H)$ be a maximal element in the isotropy lattice $\Phi_0(G; \mathscr H_{m} \setminus \{0\})$. For the symmetry group $\Gamma = D_N$ and under the assumptions \ref{A0}--\ref{A4}, if the relations
\[
4 \sin^2 \left( \frac{\pi j}{N} \right) > \frac{m^2}{\beta^2} + 1 \text{  and  } 2 \nmid \dim \mathcal V_{m,j}^H,
\]
are satisfied for an odd number of $D_N$ isotypic indices $j \in \mathfrak J(N)$, then there exists a non-stationary solution $u \in \mathscr H \setminus \{0\}$ to the system \eqref{eq:system_example} with an isotropy subgroup $G_u \leq G$ satisfying $(G_u) \geq (H)$. 
\end{theorem}
 \subsection{The Special Case of $\Gamma = D_8$}\label{sec:special_case}
While the framework described in Section \ref{sec:n_pendula} applies to any number of pendula, we can explicitly identify spatio-temporal symmetries of the non-stationary solutions to \eqref{eq:system_example} by choosing any particular dihedral order. Let's examine the special case of $8$ pendula, such that our full symmetry group becomes
\[
G:= O(2) \times D_8 \times \bz_2.
\]
Since the non-triviality of the coefficient standing next to an orbit type of maximal kind $(H) \in \mathfrak M_m$ in the basic degree $\deg_{\mathcal V_{m,j}} \in A(G)$ is equivalent to the odd-dimensionality of the associated $H$-fixed point space $\mathcal V_{m,j}^H$, the full power of Theorem 
\ref{thm:N_Pendulum} can be demonstrated using a complete list of maximal elements in the isotropy lattice  $\Phi_0(G; \mathscr H_{m} \setminus \{0\})$ for some fixed $m > 0$ together with computations of the relevant basic degrees $\deg_{\mathcal V_{m,j}}$ for $j \in \{ 0,1,2,3,4 \}$. This can be achieved for $m = 1$ (and then extended to any $m \in \bn$ via $s$-folding, cf. Section \ref{sec:sfolding}) by the following GAP code: 
\begin{lstlisting}  [language=GAP, frame=single] 
# A G.A.P. Program for the computation of Maximal Orbit Types and Basic Degrees associated with the G-isotypic decomposition 
# V = V_0 \times V_1 \times V_2 \times V_3 \times V_4
LoadPackage ("EquiDeg"); 
# create groups O(2)xD8xZ2
o2:=OrthogonalGroupOverReal(2);
d8:=pDihedralGroup(8);
z2:=pCyclicGroup(2);
# generate D8xZ2
g1:=DirectProduct(d8,z2);
# set names for subgroup conjugacy classes in D8xZ2
SetCCSsAbbrv(g1,["Z1","Z2","D1t","D1z","D1","Z1m","Z1p","D1zt","D1pt","D1p","D2","Z4","D2t","D2zt","D2d","Z4d","D2dt","D2z","Z2p","Z4p","D4dt","D2p","D4","D2pt","D4z","D4zt","Z8","Z8d","D4t","D4d","D4p","Z8p","D8","D4pt","D8d","D8z","D8dt","D8p"]);
# generate O(2)xD8xZ2
G:=DirectProduct(o2,g1);
ccs:=ConjugacyClassesSubgroups(G);
# find Maximal Orbit Types and Basic Degrees in first O(2)-isotypic component (these characterize the maximal orbit types in all O(2)-isotypic components via s-folding)
irrs := Irr(G); 
# The D_8-isotypic components 0,1,2,3,4 are indexed in GAP as 1,9,11,13,7. 
#Looping through these indices we obtain the desired data as follows:
for i in [1,9,11,13,7] do
    PrintFormatted( "\n Computing Maximal Orbit Types in M_1,{} \n", i );
   # The maximal orbit types in the isotypic component H_{m,j} and the corresponding basic degrees are obtained via the commands MaximalOrbitTypes( irrs[m,j] ); and BasicDegree( irrs[m,j] );
    max_orbtyps := MaximalOrbitTypes( irrs[1,i] );
    basic_deg := BasicDegree( irrs[1,i] );
    View(basic_deg);
    Print("\n");
    View(max_orbtyps);
od;
Print( "Done!\n" );
\end{lstlisting}
In the GAP package EquiDeg, developed for equivariant degree computations in the Burnside ring
(cf. \cite{GAP}), the conjugacy class of an amalgamated 
subgroup (cf. Section \ref{sec:generator_classification})
\[
H = K_O {}^{\varphi_O}\times^{\varphi_\Gamma}_{L} K_\Gamma,
\]
is identified by the quadruple $(K_1,K_2,Z_1,Z_2)$ where $Z_1:= \ker \varphi_1$ and $Z_2:= \ker \varphi_2$ with the notation
\begin{align*}
%\label{not_amal_2} 
(H) =: 
(K_1 {}^{Z_1}\times^{Z_2} K_2).
\end{align*}
Using this modified amalgamated notation, we can express the output of the above program, i.e. the maximal elements in the isotropy lattices $\Phi_0(G; \mathscr H_{m,j} \setminus \{0\})$ for $j \in \{0,1,2,3,4\}$ and the corresponding basic degrees as follows:
\begin{align*}
    \mathfrak M_{m,0} &= \left\{ (D_m \times D_8^p) \right\}; \\
    \mathfrak M_{m,1} &= \left\{ (D_{4m} {}^{\bz_m}\times^{\bz_4^d} D_8^p ), \; (D_{2m} {}^{D_m}\times^{\tilde D_4^d} \tilde D_4^p ), \; (D_{2m} {}^{D_m}\times^{D_4^d} D_4^p )\right\}; \\
    \mathfrak M_{m,2} &= \left\{ (D_{2m} {}^{D_m}\times^{D_2^d} D_2^p ), \; (D_{2m} {}^{D_m}\times^{\tilde D_2^d} \tilde D_2^p ), \; (D_{8m} {}^{
    \bz_m}\times^{\bz_2^-} D_8^p )\right\}; \\
    \mathfrak M_{m,3} &= \left\{ (D_{2m} {}^{D_m}\times^{D_1^p} D_2^p ), \; (D_{2m} {}^{D_m}\times^{\tilde D_1^p} \tilde D_2^p ), \; (D_{8m} {}^{
    \bz_m}\times^{\bz_1^p} D_8^p )\right\}; \\
    \mathfrak M_{m,4} &= \left\{ (D_{2m} {}^{D_m}\times^{\tilde D_4^p} D_8^p ) \right\}, \\
\end{align*}
and 
\begin{align*}
    \deg_{\mathcal V_{m,0}} &= (G) - (D_m \times D_8^p); \\
    \deg_{\mathcal V_{1,1}} &= (G) + 2(D_{2m} {}^{\bz_m}\times^{\bz_4^d} \tilde D_4^p ) + 2(D_{2m} {}^{\bz_m}\times^{\bz_4^d} D_4^p ) + (D_{2m} {}^{D_m}\times^{\bz_4^d} \bz_4^p ) \\
    & \quad \quad -2 (D_{4m} {}^{\bz_m}\times^{\bz_4^d} D_8^p ) - (D_{2m} {}^{D_m}\times^{\tilde D_4^d} \tilde D_4^p ) - (D_{2m} {}^{D_m}\times^{ D_4^d} D_4^p ); \\
    \deg_{\mathcal V_{1,2}} &= (G) + 2(D_{2m} {}^{\bz_m}\times^{\bz_2^-} D_2^p ) + 2(D_{2m} {}^{\bz_m}\times^{\bz_2^-} \tilde D_2^p ) + (D_{2m} {}^{D_m}\times^{\bz_2^-} \bz_2^p ) \\
    & \quad \quad - (D_{2m} {}^{D_m}\times^{D_2^d} D_2^p ) - (D_{2m} {}^{D_m}\times^{\tilde D_2^d} \tilde D_2^p ) - 2(D_{8m} {}^{\bz_m}\times^{} D_8^p ); \\
    \deg_{\mathcal V_{1,3}} &= (G) + 2(D_{2m} {}^{\bz_m}\times^{\bz_1^p} D_2^p ) + 2(D_{2m} {}^{\bz_m}\times^{\bz_1^p} \tilde D_2^p ) + (D_{2m} {}^{D_m}\times^{\bz_1^p} \bz_2^p ) \\
    & \quad \quad - (D_{2m} {}^{D_m}\times^{D_1^p} D_2^p ) - (D_{2m} {}^{D_m}\times^{\tilde D_1^d} \tilde D_2^p ) - 2(D_{8m} {}^{\bz_m}\times^{\bz_1^p} D_8^p ); \\
    \deg_{\mathcal V_{m,0}} &= (G) - (D_{2m} {}^{D_m}\times^{\tilde D_4^p} D_8^p ).
\end{align*}
In addition to the notation defined above, we have adopted the following ancillary shorthand for identification of subgroups in $D_8$ generated by the rotation 
$\gamma := e^{\frac{2 \pi i}{8}}$ and the relection $\kappa$:
\begin{align*}
    & \bz_1 := \{1\}, \quad \bz_2:= \{1, \gamma^4\}, \quad \bz_4 := \{1, \gamma^2, \gamma^4, \gamma^6\}, \\
    & D_1 := \{1, \kappa\}, \quad \tilde D_1 := \{1, \kappa \gamma\}, \quad D_2:= \{1, \gamma^4, \kappa, \kappa \gamma^4 \}, \\
    & \tilde D_2:= \{1, \gamma^4, \kappa \gamma, \kappa \gamma^5\}, \quad D_4:= \{1, \gamma^2, \gamma^4, \gamma^6,  \kappa, \kappa \gamma^2, \kappa \gamma^4, \kappa \gamma^6\}, \\
    & \tilde D_4 := \{1, \gamma^2, \gamma^4, \gamma^6,  \kappa \gamma, \kappa \gamma^3, \kappa \gamma^5, \kappa \gamma^7\}, \\
    & D_8 := \{1, \gamma, \gamma^2, \gamma^3, \gamma^4, \gamma^5, \gamma^6, \gamma^7, \kappa, \kappa \gamma, \kappa \gamma^2, \kappa \gamma^3, \kappa \gamma^4, \kappa \gamma^5, \kappa \gamma^6, \kappa \gamma^7 \}
\end{align*}
Likewise, the reader can consult the following list of shorthand to identify the subgroups in $D_8 \times \mathbb{Z}_2$: 
\begin{align*}
& \bz_1^p =\{(1,1),(1,-1)\}, \quad \bz_2^-=\{(1,1),(\gamma^4,-1)\}, \quad \bz_4^d := \{(1,1), (\gamma^2,-1), (\gamma^4,1), (\gamma^6,-1)\}, \\
& D_1^p := \{(1,1), (\kappa,1), (1,-1), (\kappa,-1) \}, \quad \tilde D_1^p := \{(1,1), (\kappa \gamma ,1), (1,-1), (\kappa \gamma,-1) \}, \\
& D_2^d := \{ (1,1), (\gamma^4,-1), (\kappa,1), (\kappa \gamma^4,-1) \}, \quad \tilde D_2^d := \{ (1,1), (\gamma^4,-1), (\kappa \gamma,1), (\kappa \gamma^5,-1) \}, \\
& D_2^p := \{ (1,1), (\gamma^4,1), (\kappa,1), (\kappa \gamma^4,1), (1,-1), (\gamma^4,-1), (\kappa,-1), (\kappa \gamma^4,-1) \}, \\
& \tilde D_2^p := \{ (1,1), (\gamma^4,1), (\kappa \gamma,1), (\kappa \gamma^5,1), (1,-1), (\gamma^4,-1), (\kappa \gamma,-1), (\kappa \gamma^5,-1) \}, \\
& D_4^d := \{ (1,1), (\gamma^2,-1), (\gamma^4,1), (\gamma^6,-1), (\kappa,1), (\kappa \gamma^2,-1), (\kappa\gamma^4,1), (\kappa \gamma^6,-1) \} \\
& \tilde D_4^d := \{ (1,1), (\gamma^2,-1), (\gamma^4,1), (\gamma^6,-1), (\kappa \gamma,1), (\kappa \gamma^3,-1), (\kappa\gamma^5,1), (\kappa \gamma^7,-1) \} \\
& \tilde D_4^p := \{ (1,1), (\gamma^2,1), (\gamma^4,1), (\gamma^6,1), (\kappa \gamma,1), (\kappa \gamma^3,1), (\kappa\gamma^5,1), (\kappa \gamma^7,1), \\
& \quad \quad \quad \quad 
 (1,-1), (\gamma^2,-1), (\gamma^4,-1), (\gamma^6,-1), (\kappa \gamma,-1), (\kappa \gamma^3,-1), (\kappa\gamma^5,-1), (\kappa \gamma^7,-1), \} \\
& D_8^p := \{(1,1), (\gamma,1), (\gamma^2,1), (\gamma^3,1), (\gamma^4,1), (\gamma^5,1), (\gamma^6,1), (\gamma^7,1), (\kappa,1), (\kappa \gamma,1), (\kappa \gamma^2,1),  \\
& \quad \quad \quad \quad  (\kappa \gamma^3,1), (\kappa \gamma^4,1), (\kappa \gamma^5,1), (\kappa \gamma^6,1), (\kappa \gamma^7,1), (1,-1), (\gamma,-1), (\gamma^2,-1), (\gamma^3,-1),  \\
&\quad \quad \quad \quad  
 (\gamma^6,-1), (\gamma^7,-1), (\kappa,-1), (\kappa \gamma,-1), (\kappa \gamma^2,-1), (\kappa \gamma^3,-1), (\kappa \gamma^4,-1), (\kappa \gamma^5,-1), \\
& \quad \quad \quad \quad  (\kappa \gamma^6,-1),(\gamma^4,-1), (\gamma^5,-1), (\kappa \gamma^7,-1) \}
\end{align*}
At this point, Proposition \ref{prop:special_case} follows directly from the observations $(\rm i)$ that the fixed point set $\mathcal V_{m,j}^H$ is odd dimensional for any $(H)\in \mathfrak M_{m,j}$ and $(\rm ii)$ that the sets $\mathfrak M_{m,j}$, $m > 0$, $ j \in \{0,1,2,3,4\}$ are pairwise disjoint.
 
\appendix
\section{The $G$-Equivariant Leray-Schauder Degree} \label{sec:appendix}
Given an isometric Banach $G$-representation $\mathscr H$ of functions taking values in an orthogonal $\Gamma$-representation $V$, the $G$-equivariant Leray-Schauder degree is a topological tool used to study the solution sets associated with equations of the form
\begin{align} \label{eq:1}
    \mathscr F(u) = 0, \quad u \in \mathscr H,
\end{align}
where $\mathscr F: \mathscr H \rightarrow \mathscr H$ is an operator satisfying the following conditions:
\begin{enumerate}[label=($B_\arabic*$)]
\item\label{c1} $\mathscr F$ is a $G$-equivariant completely continuous field;
\item\label{c2} there exists a sufficiently large $R > 0$ such that $\mathscr F$ is $B_R(0)$-admissibly $G$-homotopic to the identity operator $\operatorname{Id}: \mathscr H \rightarrow \mathscr H$;
\item\label{c3} $D \mathscr F(0): \mathscr H \rightarrow \mathscr H$ exists and the operator $\operatorname{Id} - D \mathscr F(0): \mathscr H \rightarrow \mathscr H$ is a $G$-equivariant compact field;
\item\label{c4} if $D \mathscr F(0): \mathscr H \rightarrow \mathscr H$ is an isomorphism, there exists a sufficiently small $\epsilon > 0$
such that $\mathscr F$ is $B_\epsilon(0)$-admissibly $G$-homotopic to $D \mathscr F(0)$.
\end{enumerate}
Condition \ref{c1} allows the problem of the existence of solutions for equation \eqref{eq:1} to be reformulated as a question concerning the non-triviality of the degree $\gdeg(\mathscr F, B(\mathscr H))$. In turn, Conditions \ref{c2}-\ref{c4} reduce calculation of $\gdeg(\mathscr F, B(\mathscr H))$ to a Burnside ring product involving a finite number of computationally simpler $G$-basic degrees.
\vs
\noi{\bf  Equivariant notation.}
Let $G$ be a compact Lie group. For any subgroup  $H \leq G$, we denote by $(H)$ its conjugacy class,
by $N(H)$ its normalizer by $W(H):=N(H)/H$ its Weyl group in $G$. The set of all subgroup conjugacy classes in $G$ is denoted by $\Phi(G):=\{(H): H\le G\}$ and has a natural partial order defined as follows
\[
(H)\leq (K) \iff \exists_{ g\in G}\;\;gHg^{-1}\leq K.
\]
As is possible with any partially ordered set, we extend the natural order over $\Phi(G)$ to a total order, which we indicate by $<$ to differentiate the two relations. Moreover, we put $\Phi_0 (G):= \{ (H) \in \Phi(G) : \text{$W(H)$  is finite}\}$ and, for any $(H),(K) \in \Phi_0(G)$, we denote by $n(H,K)$ the number of subgroups $\tilde K \leq G$ with $\tilde K \in (K)$ and $H \leq \tilde K$.
\vs
Given a $G$-space $X$ with an element $x \in X$, we denote by
$G_{x} :=\{g\in G:gx=x\}$ the {\it isotropy group} of $x$
and we call $(G_{x}) \in \Phi(G)$  the {\it orbit type} of $x \in X$. Put $\Phi(G,X) := \{(H) \in \Phi_0(G)  : 
(H) = (G_x) \; \text{for some $x \in X$}\}$ and  $\Phi_0(G,X):= \Phi(G,X) \cap \Phi_0(G)$. For a subgroup $H\leq G$, the subspace $
X^{H} :=\{x\in X:G_{x}\geq H\}$ is called the {\it $H$-fixed-point subspace} of $X$. If $Y$ is another $G$-space, then a continuous map $f : X \to Y$ is said to be {\it $G$-equivariant} if $f(gx) = gf(x)$ for each $x \in X$ and $g \in G$.
\vs
\noi{\bf The Burnside ring.}
The free $\mathbb{Z}$-module $A(G) := \mathbb{Z}[\Phi_0(G)]$ has a natural ring structure when equipped with the multiplicative operation defined, for any pair of generators $(H),(K) \in \Phi_0(G)$, as follows
\begin{align*} 
%\label{def:burnside_product}
    (H) \cdot (K) := \sum\limits_{(L) \in \Phi_0(G)} n_L(L), 
\end{align*}
and where the coefficients $n_L \in \mathbb{Z}$ are given by the recurrence formula
\begin{align*} %\label{def:recurrence_formula_coefficients_burnside_product}
    n_L := \frac{n(L,H) |W(H)| n(L,K) |W(K)| - \sum_{(\tilde L) > (L)} n_{\tilde L} n(L,\tilde L) |W(\tilde L)|}{|W(L)|}.
\end{align*}
Any {\it Burnside ring} element $a \in A(G)$ can be expressed as a formal sum over some finite number of generator elements  
\[
a = n_1(H_1) + n_2(H_2) + \cdots + n_N(H_N),
\]
and we use the notation
\[
\operatorname{coeff}^H(a) = n_H,
\]
to specify the integer coefficient standing next to the generator element $(H) \in \Phi_0(G)$.
\vs
\noi{\bf An Axiomatic Construction of the $G$-equivariant Leray-Schauder-Equivariant degree.}
Let $\mathscr E$ be any isometric Banach $G$-representation. A map $f: \mathscr E \rightarrow \mathscr E$ is said to be a completely continuous $G$-equivariant field if it can be expressed in the form
\[
f(x) = x - F(x),
\]
for some compact $G$-equivariant map $F: \mathscr E \rightarrow \mathscr E$. Moreover, given an open bounded $G$-invariant set $\Om \subset \mathscr E$, a completely continuous $G$-equivariant field is said to be $\Om$-admissible and the pair $(f,\Om)$ is called an admissible $G$-pair if $f(x) \neq 0$ for all $x \in \partial \Om$. We denote by $\mathscr M^G(\mathscr E)$ the set of all admissible $G$-pairs in $\mathscr E$ and by $\mathscr M^G$ the set of all admissible $G$-pairs defined by taking a union over all isometric Banach $G$-representations as follows
\[
\mathscr M^G := \bigcup_{\mathscr E} \mathscr M^G(\mathscr E).
\]
The $G$-equivariant Leray-Schauder degree is defined as the unique map $\gdeg: \mathscr M^G \rightarrow A(G)$ that assigns to every admissible $G$-pair $(f,\Omega)$ the Burnside ring element
	\begin{align}
		\label{eq:G-deg0}\gdeg(f,\Omega)=\sum_{(H) \in \Phi_0(G)}%
		{n_{H}(H)},
	\end{align}
satisfying the four degree axioms
	\begin{itemize}
		\item[] \textbf{(Existence)} If  $n_{H} \neq0$ for some $(H) \in \Phi_0(G)$ in \eqref{eq:G-deg0}, then there
		exists $x\in\Omega$ such that $f(x)=0$ and $(G_{x})\geq(H)$.

		\item[] \textbf{(Additivity)} 
  For any two  disjoint open $G$-invariant subsets
  $\Omega_{1}$ and $\Omega_{2}$ with
		$f^{-1}(0)\cap\Omega\subset\Omega_{1}\cup\Omega_{2}$, one has
		\begin{align*}
			\gdeg(f,\Omega)=\gdeg(f,\Omega_{1})+\gdeg
			(f,\Omega_{2}).
		\end{align*}

		\item[] \textbf{(Homotopy)} For any 
  $\Omega$-admissible $G$-homotopy, $h:[0,1]\times V\to V$, one has
		\begin{align*}
			\gdeg(h_{t},\Omega)=\mathrm{constant}.
		\end{align*}

		\item[] \textbf{(Normalization)}
  For any open bounded neighborhood of the origin in an isometric Banach $G$-representation $\mathscr E$ with the identity operator $\id:\mathscr E \rightarrow \mathscr E$, one has
		\begin{align*}
			\gdeg(\id,\Omega)=(G).
		\end{align*}
	\end{itemize}
The following are two additional properties of the map $\gdeg$ which can be derived from the four axiomatic properties defined above (cf. \cite{AED}, \cite{book-new}):		
\begin{itemize}
		\item[] {\textbf{(Multiplicativity)}} For any $(f_{1},\Omega
		_{1}),(f_{2},\Omega_{2})\in\mathcal{M} ^{G}$,
		\begin{align*}
			\gdeg(f_{1}\times f_{2},\Omega_{1}\times\Omega_{2})=
			\gdeg(f_{1},\Omega_{1})\cdot \gdeg(f_{2},\Omega_{2}),
		\end{align*}
		where the multiplication `$\cdot$' is taken in the Burnside ring $A(G )$.

		\item[] \textbf{(Recurrence Formula)} For an admissible $G$-pair
		$(f,\Omega) \in \mathscr M^G(\mathscr E)$, the $G$-degree \eqref{eq:G-deg0} can be computed using the
		following Recurrence Formula:
		\begin{equation}
			\label{eq:RF-0}n_{H}=\frac{\deg(f^{H},\Omega^{H})- \sum_{(K)>(H)}{n_{K}\,
					n(H,K)\, \left|  W(K)\right|  }}{\left|  W(H)\right|  },
		\end{equation}
		where $\left|  X\right|  $ stands for the number of elements in the set $X$
		and $\deg(f^{H},\Omega^{H})$ is the Brouwer degree of the map $f^{H}%
		:=f|_{\mathscr E^{H}}$ on the set $\Omega^{H}\subset \mathscr E^{H}$.
	\end{itemize}
\noi{\bf Computation of the $G$-equivariant Leray-Schauder degree.} 
Let's denote by $\{ \mathcal U_i \}_{i \in \mathbb{N}}$ the set of all irreducible $G$-representations and define the $i$-th basic degree as follows:
\begin{align*}
\deg_{\mathcal{U}_{i}}:=\gdeg(-\id,B(\mathcal{U} _{i})).
\end{align*} 
Given any isometric Banach $G$-representation with a $G$-isotypic decomposition
\[
\mathscr E = \bigoplus_{i \in \mathbb{N}} \mathscr E_i,
\]
and any bounded $G$-equivariant linear isomorphism $T:\mathscr E\to \mathscr E$, the Multiplicativity and Homotopy properties of the $G$-equivariant Leray-Schauder degree, together with Schur's Lemma implies
\begin{align*}
%\label{eq:prod-prop}
  \gdeg(\id - T,B(\mathscr E))=\prod_{i \in \mathbb{N}} \gdeg
	(T_{i},B(\mathscr E_{i}))= \prod_{i \in \mathbb{N}}\prod_{\mu\in\sigma_{-}(T)} \left(
	\deg_{\mathcal{U} _{i}}\right)  ^{m_{i}(\mu)}%  
\end{align*}
where $T_{i}=\id - T|_{\mathscr E_{i}}: \mathscr E_i \rightarrow \mathscr E_i$ and $\sigma_{-}(T)$ denotes the real negative
spectrum of $T$. \vskip.3cm

Notice that each of the basic degrees: 
\begin{align*}
	\deg_{\mathcal{U} _{i}}=\sum_{(H) \in \Phi_0(G)}n_{H}(H),
\end{align*}
can be practically computed, using the recurrence formula  \eqref{eq:RF-0}, as follows:
\begin{align*}
n_{H}=\frac{(-1)^{\dim\mathcal{U} _{i}^{H}}- \sum_{H<K}{n_{K}\, n(H,K)\, \left|  W(K)\right|  }}{\left|  W(H)\right|  }.
\end{align*}
\vs
The following fact is well-known (see for example \cite{survey}).
\begin{lemma}
    For any irreducible $G$-representation $\mathcal U$, the basic degree $\deg_{\mathcal U} \in A(G)$ is an involutive element of the Burnside ring, i.e.
    \[
    (\deg_{\cU})^2=\deg_{\cU} \cdot \deg_{\cU}=(G).
    \]
\end{lemma}

\end{document}